\documentclass[11pt,a4paper]{article}
\usepackage[T1]{fontenc}
\usepackage[utf8]{inputenc}
\usepackage{lmodern,microtype}
\usepackage[margin=28mm]{geometry}
\usepackage{amsmath,amssymb,amsthm,mathtools}
\usepackage{enumitem}
\usepackage{indentfirst}
\usepackage{xcolor}
\usepackage{tikz}
\usepackage[colorlinks=true,linkcolor=blue,citecolor=blue,urlcolor=blue]{hyperref}
\usepackage[noabbrev,capitalise]{cleveref}
\hypersetup{pdftitle={Unique Positive Extremals and Sharp Stability Constants for Singular Moser-Trudinger Inequalities},
            pdfauthor={Kaiwen Guo and Yanjun Liu}}
\numberwithin{equation}{section}
\newtheorem{theorem}{Theorem}[section]
\newtheorem{proposition}[theorem]{Proposition}
\newtheorem{lemma}[theorem]{Lemma}
\newtheorem{corollary}[theorem]{Corollary}
\theoremstyle{definition}

\theoremstyle{remark}
\newtheorem{remark}[theorem]{Remark}
\makeatletter
\def\cv@theoremstep[#1]#2{%
  \let\cv@savedstep\cref@old@refstepcounter
  \let\cref@old@refstepcounter\Hy@theorem@refstepcounter
  \refstepcounter[#1]{#2}%
  \let\cref@old@refstepcounter\cv@savedstep}
\AtBeginDocument{%
  \@ifundefined{Hy@theorem@refstepcounter}{}{%
    \patchcmd{\cref@thmoptarg}
      {\refstepcounter}{\cv@theoremstep}{}%
      {\PackageWarning{manuscript}{Theorem hyperlink compatibility patch not applied}}%
    \patchcmd{\cref@thmnoarg}
      {\refstepcounter}{\Hy@theorem@refstepcounter}{}%
      {\PackageWarning{manuscript}{Numbered theorem hyperlink patch not applied}}}}
\makeatother
\newcommand{\R}{\mathbb R}
\newcommand{\M}{\mathcal M}
\newcommand{\ip}[2]{\langle #1,#2\rangle}
\newcommand{\norm}[1]{\lVert #1\rVert}
\DeclareMathOperator{\dist}{dist}
\DeclareMathOperator{\spanop}{span}
\DeclareMathOperator{\Id}{Id}
\setlist[enumerate]{leftmargin=*,itemsep=3pt,topsep=5pt}
\title{Unique Positive Extremals and Sharp Stability Constants\newline
for Singular Moser-Trudinger Inequalities}
\author{Kaiwen Guo\textsuperscript{a}, Yanjun Liu\textsuperscript{b,*}\\[1ex]
\small \textsuperscript{a}Mathematical Research Center, Shandong University,\\
\small Jinan, 250100, P.~R.~China\\
\small \textsuperscript{b}School of Mathematical Sciences, Chongqing Normal University,\\
\small Chongqing, 401331, P.~R.~China}
\date{}

\begin{document}
\maketitle
\begingroup
\renewcommand{\thefootnote}{\fnsymbol{footnote}}
\footnotetext[0]{\raggedright Email: \texttt{kwguo@mail.nankai.edu.cn} (Kaiwen Guo).}
\footnotetext[1]{\raggedright Corresponding author: \texttt{liuyj@mail.nankai.edu.cn} (Yanjun Liu).}
\endgroup
\vspace{-1em}
\begin{abstract}
We combine spectral perturbation with orthogonal transformations in
$H_0^1(\Omega)$ to study singular Moser-Trudinger inequalities under
the energy constraint $\int_\Omega|\nabla u|^2\,dx\le1$, where $\Omega$
is a smooth bounded planar domain containing the origin. For the weight $|x|^{-2\beta}$,
$0\le\beta<1$, the positive maximizer is unique for small $\alpha>0$,
and the sharp global stability constant satisfies
\[
 C_\alpha=\frac{\alpha}{2}
   \left(\frac1{\mu_1}-\frac1{\mu_2}\right)+O(\alpha^2),
 \qquad \alpha\downarrow0,
\]
where $\mu_1<\mu_2$ are the first two weighted Dirichlet eigenvalues.
We compute the second asymptotic coefficient by a weighted quartic
maximization on the second eigenspace, even when the second eigenvalue
is multiple. An exact
reduction to the sphere orthogonal to the maximizer identifies the
attaining functions and yields a limiting global-to-local ratio of
$1/2$. On the disk, the energy derivative with respect to the central
value characterizes radial constrained degeneracy and yields finitely
many exceptional normalized parameters. Combined with known critical
uniqueness and nondegeneracy, this gives uniform quadratic stability
near criticality, while compactness and analytic reduction yield
finite-power stability throughout the admissible range.
\end{abstract}



\section{Introduction and main results}\label{sec:intro}

Throughout the paper, $\Omega\subset\R^2$ is a bounded connected domain with
$C^2$ boundary, $0\in\Omega$, and $\beta\in[0,1)$ is fixed.
Dependence on $\Omega$ and $\beta$ is suppressed unless needed. Set
\[
 \ip{u}{v}=\int_\Omega\nabla u\cdot\nabla v\,dx,\qquad
 \norm{u}^2=\ip{u}{u}.
\]
All orthogonality and distances refer to this Dirichlet inner
product on $H_0^1(\Omega)$, and we consider
\begin{equation}\label{eq:problem}
 S_\alpha=\sup_{\substack{u\in H_0^1(\Omega)\\\norm{u}\le1}}
              \int_\Omega |x|^{-2\beta}e^{\alpha u^2}\,dx,
 \qquad 0<\alpha\le4\pi(1-\beta).
\end{equation}
Write $J_\alpha(u)=\int_\Omega |x|^{-2\beta}e^{\alpha u^2}\,dx$ for
the functional. Positive maximizers solve the normalized equation
\begin{equation}\label{eq:EL}
 \begin{cases}
 -\Delta u=\lambda |x|^{-2\beta} u e^{\alpha u^2}&\text{in }\Omega,\\
 u=0&\text{on }\partial\Omega,\\
 \displaystyle\int_\Omega|\nabla u|^2\,dx=1,
 \end{cases}
 \qquad
 \lambda=\left(\int_\Omega |x|^{-2\beta} u^2e^{\alpha u^2}\,dx\right)^{-1}.
\end{equation}
We bound the deficit $S_\alpha-J_\alpha(u)$ in terms of distance
to maximizers, compute the sharp quadratic constant and compare it with
the constrained second variation.

The classical Moser-Trudinger inequality \cite{Trudinger,Moser}
gives the critical exponential embedding, and Adimurthi--Sandeep
\cite{AS} established its singular counterpart with critical exponent
$4\pi(1-\beta)$. Existence of extremals was developed by
Carleson--Chang and Flucher in the unweighted planar setting
\cite{CC,Flucher}, and by Csat\'o--Roy and Yang--Zhu for singular
weights \cite{CR,YZ}, with related singular embeddings on surfaces
studied by Iula--Mancini \cite{IM}. These questions also arise for
anisotropic energies: Zhou--Zhou \cite{ZZ19} proved concentration
compactness and existence of extremals for the anisotropic Dirichlet
norm. Liu established the singular concentration--compactness
principle \cite{Liu20} and studied anisotropic inequalities with
exact growth and their maximizers \cite{Liu22}, while Lu--Shen--Xue--Zhu
\cite{LSXZ24} used weighted anisotropic isoperimetric inequalities
to establish extremals in the singular case. In this direction, Guo--Liu obtained sharp singular
inequalities in the entire space \cite{GL24}, treated $L^p$ norm
perturbations on bounded domains \cite{GL26Lp}, and studied
existence and Wulff symmetry of extremals \cite{GL26Wulff}.

Quantitative stability controls the distance to extremals through
the deficit. Following the
remainder-term problems of Brezis--Lieb \cite{BL}, Bianchi--Egnell
\cite{BE} proved this principle for the Sobolev inequality, with
later developments covering quantitative and gradient-distance estimates
\cite{CFMP,FN,Neumayer}, the sharp distance exponent
$\max\{2,p\}$ of Figalli--Zhang \cite{FZ}, and fractional Sobolev
remainders of Chen--Frank--Weth \cite{CFW}. For bounds on the
stability constants, Carlen \cite{Carlen} developed a duality
method, Dolbeault--Esteban--Figalli--Frank--Loss \cite{DEFFL}
obtained explicit Sobolev bounds with sharp dimensional dependence
and a dimension-free Gaussian log-Sobolev estimate, and
Chen--Lu--Tang \cite{CLT24,CLT25} established explicit bounds for
Hardy--Littlewood--Sobolev and fractional Sobolev stability, as well
as global conformal log-Sobolev stability on the sphere.

K\"onig \cite{Konig23,Konig25} showed that the sharp global
Bianchi--Egnell stability constant is strictly smaller than the local
constant determined by the second variation and is attained away
from the extremal manifold.
Related attainment results concern Yamabe stability under two
threshold conditions \cite{AKRW26}, logarithmic Sobolev stability
\cite{DS26}, and conformal log-Sobolev stability, together with its
fractional Sobolev limit \cite{CW26}. At degenerate extremals,
compactness and analytic reduction can yield a finite-power
estimate, as in Engelstein--Neumayer--Spolaor \cite{ENS}, while sharp
fourth-order remainders occur in the Sobolev examples of Frank
\cite{Frank22} and in the Caffarelli--Kohn--Nirenberg inequality
along the Felli--Schneider curve studied by Frank--Peteranderl
\cite{FP24}.

At the exponential endpoint, Chen--Lu--Tang \cite{CLT23} studied
local stability for spherical log-Sobolev and Moser--Onofri
inequalities. For the unweighted planar Dirichlet problem,
Andrade--de~Oliveira--do~\'O--Macedo--Ratzkin \cite{Andrade}
studied quadratic gradient stability for small parameters and the
full disk parameter range, while Li--Sun \cite[Theorem~C]{LiSun}
proved critical disk uniqueness up to sign, constrained coercivity
and global quadratic stability through an exact dual variational
principle and rigorous computer-assisted estimates. Our radial
criterion in \cref{prop:shooting} treats the norm constraint
directly, as explained in \cref{rem:constrained-coercivity}.

For small parameters on general domains, we prove uniqueness of the
positive maximizer and compute the first two coefficients of the sharp
global stability constant, with the second determined by a weighted
quartic maximization on the second eigenspace, allowing a multiple
second eigenvalue. A uniform
perturbation argument gives the limiting global-to-local ratio
$1/2$ and an exact formula on the unit sphere orthogonal to the
positive maximizer, and every attaining function changes sign, has
distance $\sqrt2$ from the extremal pair, and approaches the
quartic maximizers in the second eigenspace. For a simple second
eigenvalue, we also obtain a unique attaining pair and its first
correction.

On the disk, a radial change of variables transfers critical
uniqueness and radial coercivity from \cite{LiSun}, while a
separate estimate controls nonradial perturbations. We characterize
radial constrained degeneracy by the vanishing energy derivative
with respect to the central value, obtaining finitely many
exceptional normalized parameters and, by continuation of the
critical maximizer, uniqueness and uniform quadratic stability near
criticality. At fixed $\alpha/(1-\beta)$, the rescaled sharp constant
is nondecreasing in $\beta$, with uniform order
$\alpha/(1-\beta)^2$ for small normalized exponents and near the
critical endpoint. Concentration gaps and analytic reduction also
give finite-power stability at criticality on smooth general
domains.

Define
\begin{equation}\label{eq:distance}
 \M_\alpha=\{v\in H_0^1(\Omega):\norm{v}\le1,\ \int_\Omega |x|^{-2\beta}e^{\alpha v^2}\,dx=S_\alpha\}.
\end{equation}
Every maximizer has norm one, and $\M_\alpha=-\M_\alpha$.
We write $\dist(u,\M_\alpha)=\inf_{v\in\M_\alpha}\norm{u-v}$.
Let $\mu_1<\mu_2$ be the first two weighted Dirichlet eigenvalues.
The first is simple, and its positive normalized eigenfunction $\phi$ satisfies
\begin{equation}\label{eq:eigenproblem}
 \begin{gathered}
 -\Delta\phi=\mu_1|x|^{-2\beta}\phi,\qquad \phi\in H_0^1(\Omega),\\
 \norm{\phi}=1,\qquad \int_\Omega |x|^{-2\beta}\phi^2\,dx=\frac1{\mu_1}.
 \end{gathered}
\end{equation}
All constants below may depend on $\Omega$ and the fixed value of $\beta$.
There is a constant $\alpha_0\in(0,4\pi(1-\beta))$ such that the
following three theorems hold for every $0<\alpha<\alpha_0$.

\begin{theorem}\label{thm:unique}
For $0<\alpha<\alpha_0$,
\[
 \M_\alpha=\{u_\alpha,-u_\alpha\},\qquad u_\alpha>0\text{ in }\Omega.
\]
The pair $(u_\alpha,\lambda_\alpha)$ in \eqref{eq:EL} extends smoothly to
$\alpha=0$ with $(u_0,\lambda_0)=(\phi,\mu_1)$. More precisely,
\begin{align}
 u_\alpha&=\phi+\alpha\psi+O(\alpha^2),\label{eq:branch-exp}\\
 \lambda_\alpha&=\mu_1-\alpha\mu_1^2\int_\Omega |x|^{-2\beta}\phi^4\,dx+O(\alpha^2),\label{eq:lambda-exp}\\
 S_\alpha&=\int_\Omega |x|^{-2\beta}\,dx+\frac{\alpha}{\mu_1}
                 +\frac{\alpha^2}{2}\int_\Omega |x|^{-2\beta}\phi^4\,dx+O(\alpha^3).\label{eq:sup-exp}
\end{align}
All expansions are taken as $\alpha\downarrow0$, with the first holding in
$H_0^1(\Omega)$. The correction $\psi$ is the unique weak solution of
\begin{equation}\label{eq:psi}
 \begin{cases}
 \begin{aligned}
 (-\Delta-\mu_1|x|^{-2\beta})\psi
   &=\mu_1|x|^{-2\beta}\phi^3\\
   &\quad-\mu_1^2|x|^{-2\beta}\phi
                 \int_\Omega |x|^{-2\beta}\phi^4\,dx
 \end{aligned}
            &\text{in }\Omega,\\
 \psi=0&\text{on }\partial\Omega,\\
 \ip{\psi}{\phi}=0.
 \end{cases}
\end{equation}
\end{theorem}

\begin{theorem}\label{thm:sharp}
For $0<\alpha<\alpha_0$, define
\begin{equation}\label{eq:bestconstant}
 C_\alpha
 =\inf_{\substack{u\in H_0^1(\Omega),\ \norm{u}\le1\\u\notin\M_\alpha}}
       \frac{S_\alpha-\int_\Omega |x|^{-2\beta}e^{\alpha u^2}\,dx}{\dist(u,\M_\alpha)^2}.
\end{equation}
Then $C_\alpha>0$ and
\begin{equation}\label{eq:best-exp}
 C_\alpha
  =\frac{\alpha}{2}\left(\frac1{\mu_1}-\frac1{\mu_2}\right)
        +O(\alpha^2).
\end{equation}
In particular, every $u\in H_0^1(\Omega)$ with $\norm{u}\le1$ satisfies
\begin{equation}\label{eq:stability}
 S_\alpha-\int_\Omega |x|^{-2\beta}e^{\alpha u^2}\,dx\ge C_\alpha
            \min\{\norm{u-u_\alpha}^2,\norm{u+u_\alpha}^2\}.
\end{equation}
Let $E_2$ be the weighted Dirichlet eigenspace corresponding to $\mu_2$.
The more precise expansion is
\begin{equation}\label{eq:second-best}
\begin{aligned}
 C_\alpha
 &=\frac{\alpha}{2}\left(\frac1{\mu_1}-\frac1{\mu_2}\right)\\
 &\quad+\frac{\alpha^2}{4}
 \left(\int_\Omega |x|^{-2\beta}\phi^4\,dx
       -\max_{\substack{w\in E_2\\\norm{w}=1}}
           \int_\Omega |x|^{-2\beta}w^4\,dx\right)
       +O(\alpha^3).
\end{aligned}
\end{equation}
The formula holds for both simple and multiple second eigenvalues.
The infimum in \eqref{eq:bestconstant} is unchanged if $\norm{u}\le1$
is replaced by $\norm{u}=1$.
\end{theorem}

The next result identifies the functions attaining the sharp constant
through a maximization problem orthogonal to $u_\alpha$.

\begin{theorem}\label{thm:attainment}
For $0<\alpha<\alpha_0$,
\begin{equation}\label{eq:exact-constant}
 C_\alpha=\frac12\left(S_\alpha-
       \max_{\substack{\norm{w}=1\\w\perp u_\alpha}}J_\alpha(w)\right).
\end{equation}
The set of all functions attaining \eqref{eq:bestconstant} is
\begin{equation}\label{eq:attaining-set}
 \mathcal A_\alpha=\left\{w:\norm{w}=1,\ w\perp u_\alpha,\
                  J_\alpha(w)=S_\alpha-2C_\alpha\right\}.
\end{equation}
It is nonempty and compact in $H_0^1(\Omega)$. Every element changes
sign and has distance $\sqrt2$ from $\M_\alpha$. Let
\[
 \mathcal K_2=\operatorname*{arg\,max}_{\substack{w\in E_2\\\norm{w}=1}}
                   \int_\Omega |x|^{-2\beta}w^4\,dx.
\]
As $\alpha\downarrow0$,
\begin{equation}\label{eq:attaining-selection}
 \sup_{w\in\mathcal A_\alpha}\dist(w,E_2)=O(\alpha),\qquad
 \sup_{w\in\mathcal A_\alpha}\dist(w,\mathcal K_2)\longrightarrow0.
\end{equation}
\end{theorem}

After dividing the deficit by $\alpha$, the limiting local coefficient
is $\mu_1^{-1}-\mu_2^{-1}$, whereas the global quotient is half this
value at unit vectors in the second eigenspace, which have distance
$\sqrt2$ from $\{\phi,-\phi\}$.

To formulate the disk result, let $v$ be a maximizer and let
$\lambda_v$ be its multiplier in \eqref{eq:EL}. Define
\begin{equation}\label{eq:Q}
 Q_v[h]=\norm{h}^2-\lambda_v\int_\Omega
                  |x|^{-2\beta}(1+2\alpha v^2)e^{\alpha v^2}h^2\,dx.
\end{equation}
This form is nonnegative on
$v^\perp=\{h\in H_0^1(\Omega):\ip{v}{h}=0\}$, the tangent space to
the unit sphere at $v$. Coercivity on $v^\perp$ means that
$Q_v[h]\ge c\norm{h}^2$ for all $h\perp v$ and some $c>0$.
On the disk $D=\{x\in\R^2:|x|<1\}$, radial constrained
nondegeneracy means
\begin{equation}\label{eq:RN}
 h\in H_0^1(D)\text{ radial},\qquad h\perp v,\qquad Q_v[h]=0
                 \quad\Longrightarrow\quad h=0,
\end{equation}
for every positive maximizer $v$.

\begin{theorem}
\label{thm:disk}
Let $\Omega=D$, $0\le\beta<1$, and $0<\alpha\le4\pi(1-\beta)$.
Every maximizing sequence has a strongly convergent subsequence in
$H_0^1(D)$. There exist
$C>0$ and a finite exponent $p\ge2$, depending on $\alpha$ and $\beta$,
such that
\begin{equation}\label{eq:finitepower}
 S_\alpha-\int_\Omega |x|^{-2\beta}e^{\alpha u^2}\,dx\ge C \dist(u,\M_\alpha)^p,
 \qquad \norm{u}\le1.
\end{equation}
If \eqref{eq:RN} holds at every positive maximizer, one may take
$p=2$. At each maximizer, the quadratic form is coercive on the
subspace of functions with zero angular average, independently of
\eqref{eq:RN}.
\end{theorem}

For radial functions the change of variables
\begin{equation}\label{eq:radial-transform}
 s=r^{1-\beta},\qquad U(s)=\sqrt{1-\beta}\,v(r)
\end{equation}
preserves the Dirichlet norm and radial constrained quadratic form,
and transforms the functional into
$(1-\beta)^{-1}\int_D e^{\alpha U^2/(1-\beta)}\,dx$.
At $\alpha=4\pi(1-\beta)$ this is the ordinary critical disk problem,
so \cite[Theorem~C]{LiSun} gives the following result.

\begin{corollary}\label{cor:critical}
Let $D$ be the unit disk and $0\le\beta<1$. Let $u_*>0$ maximize
$\int_D e^{4\pi u^2}\,dx$ over $u\in H_0^1(D)$ with $\norm{u}\le1$.
Then
\begin{equation}\label{eq:critical-profile}
 \begin{aligned}
 v_\beta(r)&=\frac{u_*(r^{1-\beta})}{\sqrt{1-\beta}},\qquad
 \M_{4\pi(1-\beta)}=\{v_\beta,-v_\beta\},\\
 S_{4\pi(1-\beta)}&=\frac1{1-\beta}\int_D e^{4\pi u_*^2}\,dx.
 \end{aligned}
\end{equation}
The constrained Hessian at $v_\beta$ is coercive on its full tangent
space. Consequently, there is $c_\beta>0$ such that
\begin{equation}\label{eq:critical-stability}
 S_{4\pi(1-\beta)}-\int_D |x|^{-2\beta}e^{4\pi(1-\beta)u^2}\,dx
 \ge c_\beta\min_{\sigma=\pm1}\norm{u-\sigma v_\beta}^2,
 \qquad \norm{u}\le1.
\end{equation}
The exponent two is sharp.
\end{corollary}

The next theorem describes the disk problem across the parameter range,
with an exceptional set consisting of parameters with a degenerate maximizer.

\begin{theorem}\label{thm:parameters}
Let $D$ be the unit disk. There is a finite set
$\Sigma\subset(0,4\pi)$, possibly empty and independent of $\beta$,
with the following properties for every $0\le\beta<1$ and
$0<\alpha\le4\pi(1-\beta)$.
\begin{enumerate}[label=\textup{(\roman*)}]
\item The maximizing set $\M_\alpha$ on $D$ is finite. Every maximizer
has a coercive constrained Hessian if and only if $\alpha/(1-\beta)\notin\Sigma$.
At a degenerate maximizer its constrained kernel is one-dimensional
and consists of radial functions.
\item A global quadratic deficit estimate in the Dirichlet distance
holds if and only if $\alpha/(1-\beta)\notin\Sigma$. If $\alpha/(1-\beta)\in\Sigma$, the finite-power estimate in
\cref{thm:disk} holds, and every valid exponent is at least four.
\item There is $\delta\in(0,4\pi)$, independent of $\beta$, such that
for $(4\pi-\delta)(1-\beta)\le\alpha\le4\pi(1-\beta)$ the positive maximizer
$v_\alpha$ is unique and depends real analytically on $\alpha$.
For each fixed $\beta$ there are $q_\beta,c_\beta>0$, independent of
$\alpha$ in this interval, such that
\begin{align}
 Q_{v_\alpha}[h]&\ge q_\beta\norm{h}^2
       &&\text{for }h\perp v_\alpha,\label{eq:uniform-end-Q}\\
 S_\alpha-\int_D|x|^{-2\beta}e^{\alpha u^2}\,dx
 &\ge c_\beta\min_{\sigma=\pm1}\norm{u-\sigma v_\alpha}^2
       &&\text{for }\norm{u}\le1.\label{eq:uniform-end-stability}
\end{align}
\end{enumerate}
\end{theorem}

Section~\ref{sec:prelim} develops the spectral and exponential estimates
and proves \cref{thm:unique}, while Section~\ref{sec:stability} uses
an abstract perturbation argument to prove \cref{thm:sharp,thm:attainment}
and describe the attaining functions.
Section~\ref{sec:disk} proves \cref{thm:disk,cor:critical,thm:parameters},
extends finite-power stability to general domains, including the smooth
critical case, and compares constants for different singular weights.

\section{Spectral estimates and uniqueness of maximizers}\label{sec:prelim}

\subsection{Spectral and exponential estimates}

\begin{lemma}\label{lem:weighted}
For each finite $q\ge1$, the embedding $H_0^1(\Omega)\hookrightarrow L^q(\Omega,|x|^{-2\beta}\,dx)$
is compact. The operator $T_0:H_0^1(\Omega)\to H_0^1(\Omega)$ determined by
\begin{equation}\label{eq:T}
 \ip{T_0u}{h}=\int_\Omega |x|^{-2\beta}uh\,dx
\end{equation}
is compact, positive, injective, and self-adjoint. Its two largest eigenvalues
are $\mu_1^{-1}>\mu_2^{-1}$, and the largest is simple.
\end{lemma}
\begin{proof}
Choose $s>1$ with $\beta s<1$, allowing any finite $s>1$ when
$\beta=0$, so that $|x|^{-2\beta}\in L^s(\Omega)$ and
\[
 \int_\Omega |x|^{-2\beta}|u|^q\,dx
 \le\norm{|x|^{-2\beta}}_{L^s}\norm{u}_{L^{qs'}}^q.
\]
Compact planar Sobolev embeddings prove the first assertion and the
compactness and self-adjointness of \eqref{eq:T}. The first weighted
Rayleigh quotient is attained, and the absolute value of any minimizer
is again a minimizer. Elliptic regularity with an exponent greater than
one gives continuity, and the strong maximum principle makes this
absolute value positive, so every first eigenfunction has one sign.
The first eigenspace is therefore one-dimensional, since otherwise it
would contain a nonzero function orthogonal to a positive first
eigenfunction in $L^2(|x|^{-2\beta}\,dx)$.
The spectral theorem now gives \eqref{eq:eigenproblem}.
\end{proof}

For the small-parameter proofs we use the normalized functional
\begin{equation}\label{eq:F}
 F_\alpha(u)=
 \begin{cases}
 \displaystyle\frac1\alpha\int_\Omega |x|^{-2\beta}
                  \bigl(e^{\alpha u^2}-1\bigr)\,dx,
                                                   &\quad\alpha\ne0,\\[10pt]
 \displaystyle\int_\Omega |x|^{-2\beta}u^2\,dx,      &\quad\alpha=0.
 \end{cases}
\end{equation}
This normalization gives the quadratic limit $F_0$ used below.

\begin{lemma}\label{lem:calculus}
For every fixed $R>1$, there is $\varepsilon_R>0$ such that
$(\alpha,u)\mapsto F_\alpha(u)$ is smooth for
$|\alpha|<\varepsilon_R$, $\norm{u}<R+1$. Uniformly on $\{\norm{u}\le R\}$,
\begin{equation}\label{eq:C2-exp}
 F_\alpha(u)=F_0(u)+\frac{\alpha}{2}\int_\Omega |x|^{-2\beta}u^4\,dx
                       +O_{C^2}(\alpha^2).
\end{equation}
In particular $\norm{F_\alpha-F_0}_{C^2}=O(|\alpha|)$ on the unit ball.
The map $T_\alpha$ defined by
\begin{equation}\label{eq:Talpha}
 \ip{T_\alpha(u)}{h}=\int_\Omega |x|^{-2\beta}u e^{\alpha u^2}h\,dx
\end{equation}
is smooth near $(0,\phi)$ and agrees with \eqref{eq:T} when $\alpha=0$.
\end{lemma}
\begin{proof}
Choose a larger radius $R_1>R+1$ and $\eta>0$ such that
$\eta R_1^2<4\pi(1-\beta)$. The singular Moser-Trudinger inequality gives
\[
 \sup_{\norm{u}\le R_1}\int_\Omega |x|^{-2\beta}e^{\eta u^2}\,dx<\infty.
\]
Absorb polynomial factors into an exponential with a slightly larger
coefficient below this threshold, and estimate products of directions
$h_j\in H_0^1(\Omega)$ by H\"older's inequality for the finite measure
$|x|^{-2\beta}\,dx$ and \cref{lem:weighted}. Taylor expansion, also
after one and two Fr\'echet derivatives, gives \eqref{eq:C2-exp}, and
the same bounds for any fixed number of derivatives give smoothness.
The last assertion follows from $DF_\alpha(u)[h]=2\ip{T_\alpha(u)}{h}$.
\end{proof}

\begin{lemma}\label{lem:subcompact}
For each $0<\alpha<4\pi(1-\beta)$, every maximizing sequence for
\eqref{eq:problem} has a strongly convergent subsequence in $H_0^1(\Omega)$.
Every maximizer has norm one, is bounded and continuous on
$\overline\Omega$, and has one sign. A nonnegative maximizer is positive
in $\Omega$ and satisfies \eqref{eq:EL}.
\end{lemma}
\begin{proof}
After passing to a subsequence, let $u_j\rightharpoonup u$ in
$H_0^1(\Omega)$ and almost everywhere, with $\norm{u_j}\le1$.
Choose $q>1$ with $q\alpha<4\pi(1-\beta)$, so that
$e^{\alpha u_j^2}$ is bounded in $L^q(|x|^{-2\beta}\,dx)$ and
Vitali's theorem gives convergence of the weighted exponential integrals.
For a maximizing sequence, the limit is a nonzero maximizer because
$S_\alpha>\int_\Omega |x|^{-2\beta}\,dx$, and rescaling forces its
norm to be one, giving strong convergence.

The following regularity argument also applies at criticality.
For fixed $u\in H_0^1(\Omega)$ and finite $A,q>0$, write $u=b+r$
with $b\in C_c^\infty(\Omega)$ and $\norm{r}$ arbitrarily small. The estimate
\[
 (b+r+h)^2\le2\norm{b}_\infty^2+2(r+h)^2
\]
and the Moser-Trudinger inequality show that $e^{A(u+h)^2}$ is bounded
in $L^q(\Omega)$ whenever $\norm{h}$ is sufficiently small. The
singular inequality gives the same conclusion with respect to
$|x|^{-2\beta}\,dx$. In particular, the functional is continuously
differentiable near each fixed $u$, with derivative
\[
 2\alpha\int_\Omega |x|^{-2\beta}u e^{\alpha u^2}h\,dx.
\]
The Lagrange multiplier rule on the unit sphere and testing the
resulting equation against $u$ therefore give \eqref{eq:EL}.
Since $e^{Au^2}$ belongs to every finite unweighted $L^q$ space and
$|x|^{-2\beta}\in L^s$ for some $s>1$, we have
$|x|^{-2\beta}u e^{\alpha u^2}\in L^p(\Omega)$ for some $p>1$.
Elliptic regularity gives $u\in W^{2,p}(\Omega)$, hence boundedness
and continuity on $\overline\Omega$.

Since $|u|$ is also a maximizer, it satisfies \eqref{eq:EL} and is
positive by the strong maximum principle, so continuity and connectedness
force $u$ to have one sign.
\end{proof}

\begin{lemma}\label{lem:limit}
If $\alpha_j\downarrow0$ and $u_j\in\M_{\alpha_j}$, then
$\dist(u_j,\{\phi,-\phi\})\to0$ in $H_0^1(\Omega)$. Their multipliers satisfy
$\lambda_j\to\mu_1$.
\end{lemma}
\begin{proof}
By \cref{lem:calculus} and maximality,
\[
 \int_\Omega |x|^{-2\beta}u_j^2\,dx\ge\int_\Omega |x|^{-2\beta}\phi^2\,dx-O(\alpha_j)
                       =\mu_1^{-1}-O(\alpha_j).
\]
Write $u_j=t_j\phi+z_j$ with $z_j\perp\phi$. Since $\norm{u_j}=1$,
\[
 \int_\Omega |x|^{-2\beta}u_j^2\,dx
 \le\frac{t_j^2}{\mu_1}+\frac{1-t_j^2}{\mu_2}.
\]
It follows that $t_j^2\to1$ and $z_j\to0$. Finally,
$\lambda_j^{-1}=\int_\Omega |x|^{-2\beta}u_j^2e^{\alpha_j u_j^2}\,dx\to\mu_1^{-1}$ by the
uniform exponential estimates.
\end{proof}

\subsection{Uniqueness and asymptotic expansions}\label{sec:branch}

\begin{proof}[Proof of \cref{thm:unique}]
Consider the map, with values in $H_0^1(\Omega)\times\R$,
\begin{equation}\label{eq:IFTmap}
 (\alpha,u,\lambda)\longmapsto
       \bigl(u-\lambda T_\alpha(u),\ \norm{u}^2-1\bigr).
\end{equation}
This map is smooth and vanishes at $(0,\phi,\mu_1)$, where its
derivative in $(u,\lambda)$ is
\begin{equation}\label{eq:IFTderivative}
 (h,\ell)\longmapsto
 \left((\Id-\mu_1T_0)h-\frac{\ell}{\mu_1}\phi,
                                      2\ip{\phi}{h}\right).
\end{equation}
This is an isomorphism because the second component determines the
$\phi$-component of $h$, projection of the first onto $\phi$ determines
$\ell$, and $\Id-\mu_1T_0\ge(1-\mu_1/\mu_2)\Id>0$ on $\phi^\perp$.

The implicit function theorem gives a unique smooth family near
$(\phi,\mu_1)$ and, by sign change, near $(-\phi,\mu_1)$.
By \cref{lem:limit}, these neighborhoods contain all maximizers for
small $\alpha$, so existence and symmetry give exactly the pair
$\pm u_\alpha$, with $u_\alpha>0$.

Differentiating \eqref{eq:EL} at $\alpha=0$, with
$\psi=\partial_\alpha u_\alpha|_{0}$, gives
\[
 (-\Delta-\mu_1|x|^{-2\beta})\psi
 =\lambda'_0 |x|^{-2\beta}\phi+\mu_1|x|^{-2\beta}\phi^3,
 \qquad \ip{\psi}{\phi}=0.
\]
Testing against $\phi$ gives
$\lambda'_0/\mu_1+\mu_1\int_\Omega |x|^{-2\beta}\phi^4\,dx=0$, hence
$\lambda'_0=-\mu_1^2\int_\Omega |x|^{-2\beta}\phi^4\,dx$. This gives \eqref{eq:branch-exp}, \eqref{eq:lambda-exp} and
\eqref{eq:psi}, whose right-hand side is orthogonal to $\phi$,
so the spectral gap gives a unique solution in $\phi^\perp$.

Finally,
\[
 S_\alpha=\int_\Omega |x|^{-2\beta}
       \left(1+\alpha u_\alpha^2+\frac{\alpha^2}{2}u_\alpha^4\right)\,dx
       +O(\alpha^3).
\]
Since $\int_\Omega |x|^{-2\beta}\phi\psi\,dx=\mu_1^{-1}\ip{\phi}{\psi}=0$,
the weighted quadratic integral equals $\mu_1^{-1}+O(\alpha^2)$,
and the weighted quartic integral equals
$\int_\Omega |x|^{-2\beta}\phi^4\,dx+O(\alpha)$.
This proves \eqref{eq:sup-exp}.
\end{proof}

\section{Global stability and the sharp small-parameter constant}
\label{sec:stability}

We prove \cref{thm:sharp} through \cref{prop:abstract}, starting
with a comparison of two deficits.

\begin{lemma}\label{lem:abstract-transfer}
Let $F,G$ be even $C^2$ functions on a neighborhood of the closed unit
ball in a real Hilbert space. Suppose that $v$ is a unit constrained
critical point of both functions. Then, for every unit $u$,
\begin{equation}\label{eq:abstract-transfer}
 \bigl|[F(v)-F(u)]-[G(v)-G(u)]\bigr|
 \le\frac12\left(\norm{D(F-G)(v)}
            +\sup_{B_1}\norm{D^2(F-G)}\right)
           \min_{\sigma=\pm1}\norm{u-\sigma v}^2.
\end{equation}
\end{lemma}
\begin{proof}
For $H=F-G$, constrained criticality gives $DH(v)=b\,v$ with
$b\in\R$, while $\ip{v}{u-v}=-\norm{u-v}^2/2$ for unit $u$.
Taylor's formula on the line segment from $v$ to $u$
therefore bounds $|H(u)-H(v)|$ by the right side of
\eqref{eq:abstract-transfer} with $\norm{u-v}^2$ in place of the
minimum. Apply the same argument at $-v$, using evenness.
\end{proof}

The following proposition yields \eqref{eq:second-best}.
All norms are Hilbert norms, and the $C^2$ norm includes the function
and its first two Fr\'echet derivatives.

\begin{proposition}\label{prop:abstract}
Let $H$ be a real Hilbert space and let $T$ be a compact, nonnegative,
self-adjoint operator. Suppose that its largest eigenvalue $t_1$ is
simple and that its next distinct eigenvalue satisfies $0<t_2<t_1$.
Let $\phi$ be a unit $t_1$-eigenvector, put $E_2=\ker(T-t_2\Id)$,
and set $F_0(u)=\ip{Tu}{u}$. On a fixed ball $B_R\subset H$ with
$R>1$, suppose that the even functionals $F_\varepsilon$ satisfy
\begin{equation}\label{eq:abstract-assumption}
 F_\varepsilon=F_0+\varepsilon P+O_{C^2(B_R)}(\varepsilon^2),
 \qquad \varepsilon\downarrow0,
\end{equation}
where $P$ is even and has bounded $C^2(B_R)$ norm. Assume that each
$F_\varepsilon$ attains its maximum on the unit sphere.
For all sufficiently small $\varepsilon>0$ its maximizers are
exactly $\{v_\varepsilon,-v_\varepsilon\}$, with
$\norm{v_\varepsilon-\phi}=O(\varepsilon)$ after choosing the sign.
Define its sharp sphere constant by
\[
 c_\varepsilon=
 \inf_{\substack{\norm{u}=1\\u\ne\pm v_\varepsilon}}
 \frac{F_\varepsilon(v_\varepsilon)-F_\varepsilon(u)}
      {\min_{\sigma=\pm1}\norm{u-\sigma v_\varepsilon}^2}.
\]
Then
\begin{equation}\label{eq:abstract-expansion}
 c_\varepsilon=\frac{t_1-t_2}{2}
  +\frac{\varepsilon}{2}
     \left(P(\phi)-\max_{\substack{w\in E_2\\\norm{w}=1}}P(w)\right)
  +O(\varepsilon^2).
\end{equation}
Moreover,
\begin{equation}\label{eq:abstract-exact}
 c_\varepsilon=\frac12\left(F_\varepsilon(v_\varepsilon)-
       \sup_{\substack{\norm{w}=1\\w\perp v_\varepsilon}}F_\varepsilon(w)\right),
\end{equation}
and every unit vector attaining $c_\varepsilon$ is orthogonal to
$v_\varepsilon$. The least local quadratic coefficient at $v_\varepsilon$ is
$t_1-t_2+O(\varepsilon)$. In particular the global-to-local ratio
tends to $1/2$.
\end{proposition}
\begin{proof}
Write $g=t_1-t_2>0$. If $v$ is any maximizer, comparison with
$\phi$ and \eqref{eq:abstract-assumption} gives
$t_1-\ip{Tv}{v}=O(\varepsilon)$. The spectral gap implies
$1-|\ip{v}{\phi}|^2=O(\varepsilon)$, so all maximizers lie near
$\phi$ or $-\phi$.

Write $u=(\phi+x)/(1+\norm{x}^2)^{1/2}$ with $x\perp\phi$.
In these coordinates, the Hessian of $F_0$ at zero is
$2(T-t_1\Id)\le-2g\Id$, so it remains uniformly negative on a
fixed small ball. By $C^2$ convergence, $F_\varepsilon$ is strictly
concave there and has at most one critical point. Since all maximizers
lie near $\pm\phi$, existence and evenness give exactly the pair
$\pm v_\varepsilon$.
Integrating the uniform Hessian bound from $x=0$, where the gradient
is $O(\varepsilon)$, to the critical point gives
$\norm{v_\varepsilon-\phi}=O(\varepsilon)$.

With $P_\perp=\Id-\phi\otimes\phi$, project
$\tfrac12DF_\varepsilon(v_\varepsilon)=m_\varepsilon v_\varepsilon$
onto $\phi^\perp$ and use $m_\varepsilon=t_1+O(\varepsilon)$.
Taylor expansion and \eqref{eq:abstract-assumption} then give
\begin{equation}\label{eq:abstract-branch}
 v_\varepsilon=\phi+\varepsilon\psi+O(\varepsilon^2),\qquad
 \psi=\tfrac12(t_1\Id-T)^{-1}_{\phi^\perp}P_\perp DP(\phi).
\end{equation}
The component along $\phi$ changes only by $O(\varepsilon^2)$,
because both vectors have norm one.

Let $R_\varepsilon$ be the rotation sending $\phi$ to
$v_\varepsilon$ and fixing the orthogonal complement of their span.
It satisfies
\[
 R_\varepsilon=\Id+\varepsilon A+O(\varepsilon^2),\qquad
 Au=\ip{\phi}{u}\psi-\ip{\psi}{u}\phi.
\]
The formula remains valid when $\psi=0$. Since
$\norm{F_\varepsilon\circ R_\varepsilon-F_0}_{C^2(B_1)}=O(\varepsilon)$,
\cref{lem:abstract-transfer} gives the uniform quotient comparison
\begin{equation}\label{eq:abstract-uniform}
 q_\varepsilon(u)=q_0(u)+O(\varepsilon),\qquad
 q_\varepsilon(u)=
 \frac{F_\varepsilon(v_\varepsilon)-F_\varepsilon(R_\varepsilon u)}
      {\min_{\sigma=\pm1}\norm{u-\sigma\phi}^2}.
\end{equation}
Thus $\inf q_\varepsilon=g/2+O(\varepsilon)$, by the spectral
calculation in \cref{lem:limitconstant}, with $t_1,t_2$ in place of
$\mu_1^{-1},\mu_2^{-1}$.

We first prove \eqref{eq:abstract-exact}. After a sign change, write
$u=t\phi+\sqrt{1-t^2}\,w$, where $0\le t<1$, $w\perp\phi$,
and $\norm{w}=1$. Then
\[
 q_0(t,w)=\tfrac12\bigl(t_1-\ip{Tw}{w}\bigr)(1+t).
\]
Set $G_\varepsilon=F_\varepsilon\circ R_\varepsilon-F_0$.
For $0\le t\le1/2$, differentiating
\[
 q_\varepsilon(t,w)-q_0(t,w)
 =\frac{G_\varepsilon(\phi)-G_\varepsilon(t\phi+\sqrt{1-t^2}\,w)}
        {2(1-t)}
\]
gives an $O(\varepsilon)$ bound uniform in $t,w$, because
$\norm{G_\varepsilon}_{C^1(B_1)}=O(\varepsilon)$ and the denominator
is bounded away from zero. Hence
\[
 \partial_tq_\varepsilon(t,w)
 =\tfrac12\bigl(t_1-\ip{Tw}{w}\bigr)+O(\varepsilon)
 \ge g/4>0.
\]
For $t\ge1/2$, \eqref{eq:abstract-uniform} gives
$q_\varepsilon(t,w)\ge3g/4-O(\varepsilon)$, whereas a unit vector
in $E_2$ at $t=0$ gives $g/2+O(\varepsilon)$.
Thus the infimum may be restricted to $t=0$, and every attaining
vector has $t=0$, which after rotation gives \eqref{eq:abstract-exact}
and the claimed orthogonality.

To obtain the next coefficient, choose the sign of $u$ so that
$t=\ip{u}{\phi}\ge0$, and decompose
\[
 u=t\phi+z+r,\qquad z\in E_2,\quad
 r\perp\bigl(\spanop\{\phi\}\oplus E_2\bigr).
\]
Compactness of $T$ and $t_2>0$ give $\gamma>0$ such that
$\ip{(t_2\Id-T)r}{r}\ge\gamma\norm{r}^2$. If the last
orthogonal complement is zero, omit its terms. Direct calculation gives
\begin{equation}\label{eq:abstract-split}
 q_0(u)=\frac g2(1+t)
   +\frac{\ip{(t_2\Id-T)r}{r}}{2(1-t)}.
\end{equation}
Any unit vector whose quotient is within $O(\varepsilon)$ of the
infimum therefore has $t=O(\varepsilon)$ and
$\norm{r}=O(\sqrt\varepsilon)$. In particular $t\le1/2$ and
$w=z/\norm{z}$ is defined for small $\varepsilon$.

Uniformly on the unit ball, expansion of the numerator gives
\[
 F_\varepsilon(R_\varepsilon u)
   =F_0(u)+\varepsilon H(u)+O(\varepsilon^2),\qquad
 H(u)=P(u)+2\ip{Tu}{Au}.
\]
Here $H(\phi)=P(\phi)$ and $H(w)=P(w)$ for every unit $w\in E_2$.
Moreover $H$ is Lipschitz on the unit ball, and
$\norm{u-w}\le C(t+\norm{r})$ for these unit vectors.
Since $2(1-t)$ is bounded below, it follows that
\begin{equation}\label{eq:abstract-refined}
 q_\varepsilon(u)
 =q_0(u)+\frac\varepsilon2\bigl(P(\phi)-P(w)\bigr)
       +O\bigl(\varepsilon(t+\norm{r})+\varepsilon^2\bigr).
\end{equation}
Compactness of the unit sphere of $E_2$ gives
$P_2=\max_{w\in E_2,\,\norm{w}=1}P(w)$, and testing the quotient
at a maximizer gives
$c_\varepsilon\le g/2+\varepsilon(P(\phi)-P_2)/2+O(\varepsilon^2)$.
For the lower bound, take unit vectors whose quotients are within
$\varepsilon^2$ of the infimum. Equations \eqref{eq:abstract-split} and
\eqref{eq:abstract-refined} give
\[
 q_\varepsilon(u)\ge\frac g2
 +\frac\varepsilon2\bigl(P(\phi)-P_2\bigr)
 +\frac g2t+\frac\gamma2\norm{r}^2
 -C\varepsilon(t+\norm{r})-C\varepsilon^2.
\]
For small $\varepsilon$, absorb $C\varepsilon t$ into $gt/2$
and complete the square in $\norm{r}$ to bound the remaining terms
below by $-C'\varepsilon^2$, proving \eqref{eq:abstract-expansion}.

Finally, $R_\varepsilon$ transfers the quadratic deficit to
$\phi^\perp$, where it equals $t_1\Id-T+O(\varepsilon)$ in
operator norm and has infimum $g+O(\varepsilon)$ on unit vectors.
The parametrization preserves distances to first order, so this is
the least local coefficient and gives the stated ratio.
\end{proof}

\begin{lemma}\label{lem:limitconstant}
On the unit sphere of $H_0^1(\Omega)$,
\begin{equation}\label{eq:limitingbest}
 \inf_{u\ne\pm\phi}
 \frac{\mu_1^{-1}-\int_\Omega |x|^{-2\beta}u^2\,dx}
      {\min\{\norm{u-\phi}^2,\norm{u+\phi}^2\}}
       =\frac12\left(\mu_1^{-1}-\mu_2^{-1}\right).
\end{equation}
\end{lemma}
\begin{proof}
Choose the sign of $u$ so that $t=\ip{u}{\phi}\ge0$, and write
$u=t\phi+z$, $z\perp\phi$. The spectral gap gives
\[
 \begin{gathered}
 \mu_1^{-1}-\int_\Omega |x|^{-2\beta}u^2\,dx
                  \ge(\mu_1^{-1}-\mu_2^{-1})(1-t^2),\\
 \min\{\norm{u-\phi}^2,\norm{u+\phi}^2\}=2(1-t).
 \end{gathered}
\]
Thus the quotient is at least
$(\mu_1^{-1}-\mu_2^{-1})(1+t)/2\ge(\mu_1^{-1}-\mu_2^{-1})/2$.
A normalized second eigenfunction realizes equality.
\end{proof}

\begin{lemma}
\label{lem:rotation}
There are orthogonal operators $R_\alpha$ on $H_0^1(\Omega)$ such that
$R_\alpha\phi=u_\alpha$ and $R_\alpha-\Id=O(\alpha)$ in operator norm,
and, for every $u$ on the unit sphere with $u\ne\pm\phi$,
\begin{equation}\label{eq:uniformquotient}
 \frac{F_\alpha(u_\alpha)-F_\alpha(R_\alpha u)}
      {\min_{\sigma=\pm1}\norm{u-\sigma\phi}^2}
 =\frac{F_0(\phi)-F_0(u)}
      {\min_{\sigma=\pm1}\norm{u-\sigma\phi}^2}+O(\alpha).
\end{equation}
The $O(\alpha)$ bound is independent of $u$, including as
$u\to\phi$ or $u\to-\phi$.
\end{lemma}
\begin{proof}
Take the rotation in $\spanop\{\phi,u_\alpha\}$ sending $\phi$ to
$u_\alpha$ and fixing its orthogonal complement, using the identity
when the two vectors coincide. The operator estimate follows from
$\norm{u_\alpha-\phi}=O(\alpha)$.

By \cref{lem:calculus}, $G_\alpha=F_\alpha\circ R_\alpha-F_0$
satisfies $\norm{G_\alpha}_{C^2}=O(\alpha)$ on the unit ball. Both terms have a constrained critical point at
$\phi$, so the gradient of $G_\alpha$ at $\phi$ is parallel to $\phi$
and has norm $O(\alpha)$.
For $u$ on the unit sphere,
\[
 \ip{\phi}{u-\phi}=-\tfrac12\norm{u-\phi}^2.
\]
Taylor's formula along the line segment from $\phi$ to $u$ consequently
gives
\[
 |G_\alpha(u)-G_\alpha(\phi)|\le C\alpha\norm{u-\phi}^2.
\]
By evenness the same estimate holds at $-\phi$, so choosing the
nearer sign proves \eqref{eq:uniformquotient}, uniformly even as $u$ approaches
$\phi$ or $-\phi$.
\end{proof}

\begin{lemma}\label{lem:ball}
Suppose that $\M_\alpha$ is nonempty and compact, and let $c\ge0$
be the sharp stability constant on the unit sphere. Then the sharp
constant on the closed unit ball is also $c$, and
\begin{equation}\label{eq:ball-remainder}
 S_\alpha-J_\alpha(u)\ge
 c\bigl(\dist(u,\M_\alpha)^2+1-\norm{u}^2\bigr),
 \qquad \norm{u}\le1.
\end{equation}
If $c>0$, no function of norm less than one attains the stability quotient.
\end{lemma}
\begin{proof}
Write $B=J_\alpha(0)$ and choose an orthonormal sequence $e_j$.
Compactness of $\M_\alpha$ gives
$\sup_{v\in\M_\alpha}|\ip{e_j}{v}|\to0$, so
$\dist(e_j,\M_\alpha)^2\to2$. Since $J_\alpha(e_j)\ge B$,
the sphere estimate yields $2c\le S_\alpha-B$.

For $u=tz$, $0<t\le1$ and $\norm{z}=1$, put
$h=\max_{v\in\M_\alpha}\ip{z}{v}\in[0,1]$.
Here every maximizer has norm one and $\M_\alpha=-\M_\alpha$.
Convexity of the exponential gives
\begin{align*}
 S_\alpha-J_\alpha(tz)
 &\ge t^2\bigl(S_\alpha-J_\alpha(z)\bigr)+(1-t^2)(S_\alpha-B)\\
 &\ge2ct^2(1-h)+2c(1-t^2)\\
 &=2c(1-t^2h)\ge2c(1-th)\\
 &=c\bigl(\dist(tz,\M_\alpha)^2+1-t^2\bigr).
\end{align*}
The bound $S_\alpha-B\ge2c$ covers $u=0$, completing
\eqref{eq:ball-remainder} and proving equality of the sharp constants
and nonattainment in the interior when $c>0$.
The argument also works for any equivalent Hilbert norm, using only
compactness of the unit-norm maximizer set and convexity along rays.
\end{proof}

\begin{proof}[Proof of \cref{thm:sharp}]
An orthogonal map preserves both the sphere and distances, so
\cref{lem:limitconstant,lem:rotation} imply
\[
 \inf_{\norm{u}=1,\ u\notin\M_\alpha}
 \frac{S_\alpha-\int_\Omega |x|^{-2\beta}e^{\alpha u^2}\,dx}{\alpha \dist(u,\M_\alpha)^2}
 =\frac12(\mu_1^{-1}-\mu_2^{-1})+O(\alpha).
\]
This is positive for sufficiently small $\alpha$.
For the refined expansion apply \cref{prop:abstract} to $T=T_0$,
$\varepsilon=\alpha$ and
\[
 P(u)=\tfrac12\int_\Omega|x|^{-2\beta}u^4\,dx.
\]
The required $C^2$ expansion is \cref{lem:calculus}, attainment follows
from \cref{lem:subcompact}, and $t_i=\mu_i^{-1}$ for $i=1,2$.
Multiplying \eqref{eq:abstract-expansion} by $\alpha$ proves
\eqref{eq:second-best} on the sphere, and \cref{lem:ball} extends
the conclusions to the closed ball with the same sharp constants.
\end{proof}

\begin{proof}[Proof of \cref{thm:attainment}]
Applying \eqref{eq:abstract-exact} as in the proof of
\cref{thm:sharp} and multiplying by $\alpha$ gives
\eqref{eq:exact-constant}, initially with a supremum in place
of the maximum.
For subcritical $\alpha$, \cref{lem:subcompact} gives sequential
weak continuity of $J_\alpha$ on the closed unit ball, so it attains
its maximum on $\{w:\norm{w}\le1,\ w\perp u_\alpha\}$.
This value exceeds $J_\alpha(0)$, and scaling forces each maximizer
to have norm one, proving attainment in \eqref{eq:exact-constant}.
Applied to a sequence of these maximizers, the same argument gives
a weak limit of norm one and hence strong convergence.

By \cref{prop:abstract}, every sphere minimizer of the stability
quotient is orthogonal to $u_\alpha$. Together with \cref{lem:ball}, which excludes attainment in the
interior of the ball, this proves \eqref{eq:attaining-set} and
compactness of $\mathcal A_\alpha$.
For $w\in\mathcal A_\alpha$, both distances to $\pm u_\alpha$ equal
$\sqrt2$. Testing \eqref{eq:EL} against $w$ gives
\[
 0=\ip{u_\alpha}{w}
 =\lambda_\alpha\int_\Omega |x|^{-2\beta}
                   u_\alpha e^{\alpha u_\alpha^2}w\,dx.
\]
The coefficient of $w$ is positive almost everywhere, so $w$ changes sign.

For the limiting profiles, write $R_\alpha^{-1}w=z+r$ with
$z\in E_2$ and $r\perp(\spanop\{\phi\}\oplus E_2)$.
Uniform quotient comparison gives $\norm{r}=O(\sqrt\alpha)$.
Let $\widehat z=z/\norm{z}$ and set
$P(v)=\tfrac12\int_\Omega |x|^{-2\beta}v^4\,dx$ and
$P_2=\max_{v\in E_2,\,\norm{v}=1}P(v)$.
Comparing the minimizing vector with a maximizer of $P$ in
\eqref{eq:abstract-split}--\eqref{eq:abstract-refined}, where $t=0$, gives
\[
 \gamma\norm{r}^2+\alpha\bigl(P_2-P(\widehat z)\bigr)
       \le C\alpha\norm{r}+C\alpha^2.
\]
Nonnegativity of both terms on the left gives $\norm{r}=O(\alpha)$
and $P_2-P(\widehat z)=O(\alpha)$ uniformly over $\mathcal A_\alpha$.
Together with $R_\alpha-\Id=O(\alpha)$, this yields
$\norm{w-\widehat z}=O(\alpha)$, so compactness of the unit sphere
of $E_2$ proves \eqref{eq:attaining-selection}.
\end{proof}

For a polynomial application, let $b\in L^\infty(\Omega)$ be nonnegative
and consider on the Dirichlet unit sphere
\[
 F_\varepsilon(u)=\int_\Omega u^2\,dx
                    +\varepsilon\int_\Omega b u^4\,dx.
\]
Compact Sobolev embeddings give weak continuity on the closed ball,
and radial rescaling forces every nonzero maximizer to have norm one
by increasing the first integral without decreasing the second. If $\nu_1<\nu_2$ are the first two ordinary Dirichlet
eigenvalues, $\phi$ is the normalized first eigenfunction, and $E_2$
is the second eigenspace, then the sharp sphere deficit constant is
\[
 \frac12\left(\frac1{\nu_1}-\frac1{\nu_2}\right)
 +\frac\varepsilon2\left(\int_\Omega b\phi^4\,dx
      -\max_{\substack{w\in E_2\\\norm{w}=1}}
            \int_\Omega b w^4\,dx\right)+O(\varepsilon^2).
\]

To express the local constant through the constrained Hessian, define
\[
 C_\alpha^{\mathrm{loc}}=\lim_{\delta\downarrow0}
  \inf_{\substack{\norm{u}=1\\0<\dist(u,\M_\alpha)<\delta}}
      \frac{S_\alpha-J_\alpha(u)}{\dist(u,\M_\alpha)^2}.
\]
For $h\perp u_\alpha$, set
$u=(u_\alpha+h)/(1+\norm{h}^2)^{1/2}$. Taylor expansion gives
\[
 S_\alpha-J_\alpha(u)=\frac{\alpha}{\lambda_\alpha}Q_{u_\alpha}[h]
                      +O(\norm{h}^3),\qquad
 \dist(u,\M_\alpha)^2=\norm{h}^2+O(\norm{h}^4).
\]
The remainders are uniform in the direction of $h$ for fixed
$\alpha$, and evenness gives the same value near $-u_\alpha$.
Consequently,
\begin{equation}\label{eq:local-constant}
 \begin{aligned}
 C_\alpha^{\mathrm{loc}}
   &=\frac{\alpha}{\lambda_\alpha}
        \inf_{\substack{h\perp u_\alpha\\\norm{h}=1}}Q_{u_\alpha}[h]\\
   &=\alpha(\mu_1^{-1}-\mu_2^{-1})+O(\alpha^2),\qquad
 \frac{C_\alpha}{C_\alpha^{\mathrm{loc}}}=\frac12+O(\alpha).
 \end{aligned}
\end{equation}
Indeed, \cref{lem:calculus,lem:rotation} and
$\lambda_\alpha=\mu_1+O(\alpha)$ imply, uniformly for $k\perp\phi$,
\[
 \lambda_\alpha^{-1}Q_{u_\alpha}[R_\alpha k]
 =\mu_1^{-1}\norm{k}^2-\int_\Omega|x|^{-2\beta}k^2\,dx
       +O(\alpha)\norm{k}^2.
\]
Taking the infimum over unit $k$ proves the expansion, and the ratio
follows from \cref{thm:sharp}.
For any normalized second eigenfunction $\phi_2$,
\eqref{eq:uniformquotient} and \eqref{eq:best-exp} also show that
the quotient at $R_\alpha\phi_2$ is $C_\alpha+O(\alpha^2)$.

The factor $1/2$ also has an exact two-dimensional interpretation.
Along the unit-sphere arc
$u=\cos\theta\,\phi+\sin\theta\,\phi_2$, with $0<\theta\le\pi/2$,
orthogonality and the eigenvalue equations give
\[
 \frac{F_0(\phi)-F_0(u)}
 {\bigl(\mu_1^{-1}-\mu_2^{-1}\bigr)
       \min_{\sigma=\pm1}\norm{u-\sigma\phi}^2}
 =\frac{\sin^2\theta}{2(1-\cos\theta)}
 =\frac{1+\cos\theta}{2}.
\]
As shown in \cref{fig:local-global}, the local limit at $\phi$ equals
one, whereas the global minimum equals $1/2$ and is attained at
$\phi_2$, which lies at distance $\sqrt2$ from both maximizers.

\begin{figure}[htb]
\centering
\begin{minipage}[t]{0.47\textwidth}
\centering
\small\textbf{(a) Geometry on the unit sphere}\par\medskip
\begin{tikzpicture}[x=3.2cm,y=3.2cm,font=\small,>=stealth]
 \path[use as bounding box] (-0.22,-0.23) rectangle (1.35,1.28);
 \draw[->,thin] (-0.06,0)--(1.24,0);
 \draw[->,thin] (0,-0.06)--(0,1.22);
 \draw[line width=0.9pt] (1,0) arc[start angle=0,end angle=90,radius=1];
 \draw[dashed,gray] (1,0)--(0,1);
 \node[fill=white,inner sep=2pt] at (0.45,0.55) {$\sqrt2$};
 \coordinate (u) at ({cos(25)},{sin(25)});
 \draw[densely dotted] (0,0)--(u);
 \draw[thin] (1,0)--(u);
 \draw[thin] (0.27,0) arc[start angle=0,end angle=25,radius=0.27];
 \node at (0.37,0.08) {$\theta$};
 \fill (1,0) circle (1.7pt) node[below=4pt] {$\phi$};
 \fill (0,1) circle (1.7pt) node[left=4pt] {$\phi_2$};
 \fill (u) circle (1.7pt) node[above right=2pt] {$u$};
 \node[below left=2pt] at (0,0) {$0$};
\end{tikzpicture}
\end{minipage}\hfill
\begin{minipage}[t]{0.50\textwidth}
\centering
\small\textbf{(b) Normalized deficit quotient}\par\medskip
\begin{tikzpicture}[x=2.7cm,y=3.2cm,font=\small,>=stealth]
 \path[use as bounding box] (-0.23,-0.23) rectangle (1.84,1.28);
 \draw[->,thin] (0,0)--(1.77,0) node[right] {$\theta$};
 \draw[->,thin] (0,0)--(0,1.18);
 \draw[densely dashed,gray] (0,0.5)--({pi/2},0.5);
 \draw[densely dotted,gray] ({pi/2},0)--({pi/2},0.5);
 \draw[line width=0.9pt,domain=0:90,samples=65]
       plot ({\x*pi/180},{(1+cos(\x))/2});
 \draw[fill=white,line width=0.8pt] (0,1) circle (2pt);
 \fill ({pi/2},0.5) circle (1.8pt);
 \node[anchor=west] at (0.08,1.13) {Local limit};
 \node at (1.11,0.29) {Global minimum};
 \foreach \x/\lab in {0/0,0.785398/{\pi/4},1.570796/{\pi/2}} {
   \draw (\x,0)--(\x,-0.025);
   \node[below=3pt] at (\x,0) {$\lab$};
 }
 \foreach \y/\lab in {0.5/{1/2},1/1} {
   \draw (0,\y)--(-0.025,\y);
   \node[left=3pt] at (0,\y) {$\lab$};
 }
\end{tikzpicture}
\end{minipage}
\caption{The distinction between local and global stability in the
limiting quadratic problem. The left panel shows the unit-sphere arc
in $\spanop\{\phi,\phi_2\}$ for the Dirichlet inner product.
The right panel plots the deficit-to-squared-distance quotient divided
by $\mu_1^{-1}-\mu_2^{-1}$. The open endpoint denotes the local limit
as $\theta\downarrow0$, and the filled endpoint gives the global minimum.}
\label{fig:local-global}
\end{figure}

When the second eigenvalue is simple, the functions attaining
$C_\alpha$ form a unique pair, whose first correction is explicit.

\begin{proposition}\label{prop:attaining-branch}
Suppose $E_2=\spanop\{\eta\}$ with $\norm{\eta}=1$.
For sufficiently small $\alpha>0$,
$\mathcal A_\alpha=\{w_\alpha,-w_\alpha\}$ with
$\ip{w_\alpha}{\eta}>0$. The map $\alpha\mapsto w_\alpha$
extends smoothly to $\alpha=0$, where $w_0=\eta$.
Let $X=(\spanop\{\phi,\eta\})^\perp$ and let $\xi\in X$
be the unique solution of
\begin{equation}\label{eq:attaining-correction}
 \ip{(\mu_2^{-1}\Id-T_0)\xi}{h}
       =\int_\Omega |x|^{-2\beta}\eta^3h\,dx,
       \qquad h\in X.
\end{equation}
With $\psi$ from \eqref{eq:psi},
\begin{equation}\label{eq:attaining-expansion}
 w_\alpha=\eta+\alpha\bigl(\xi-\ip{\psi}{\eta}\phi\bigr)
                  +O(\alpha^2)
       \quad\text{in }H_0^1(\Omega).
\end{equation}
\end{proposition}
\begin{proof}
Choose the rotations smoothly by setting $c_\alpha=\ip{u_\alpha}{\phi}$
and
\[
 K_\alpha h=\ip{\phi}{h}u_\alpha-\ip{u_\alpha}{h}\phi,\qquad
 R_\alpha=\Id+K_\alpha+\frac{K_\alpha^2}{1+c_\alpha}.
\]
This orthogonal map sends $\phi$ to $u_\alpha$, and
$R_0=\Id$. For $h\in X$ near zero, consider
\[
 G_\alpha(h)=F_\alpha\left(
       R_\alpha\frac{\eta+h}{\sqrt{1+\norm{h}^2}}\right).
\]
Its $X$-gradient vanishes at $(\alpha,h)=(0,0)$, and its derivative
in $h$ is $2(T_0-\mu_2^{-1}\Id)|_X$. Simplicity and isolation of
$\mu_2^{-1}$ give $\gamma>0$ such that
\[
 \ip{(\mu_2^{-1}\Id-T_0)h}{h}\ge\gamma\norm{h}^2,
 \qquad h\in X.
\]
The implicit function theorem gives a unique small critical point
$h_\alpha$, smooth in $\alpha$, and \cref{thm:attainment} with
$\mathcal K_2=\{\eta,-\eta\}$ places every attaining function in
this chart after a sign change and rotation.
Thus the attaining pair is generated by
$w_\alpha=R_\alpha(\eta+h_\alpha)/\sqrt{1+\norm{h_\alpha}^2}$.

Differentiating the rotation gives
$R'_0z=\ip{\phi}{z}\psi-\ip{\psi}{z}\phi$.
For $z\perp\phi$, the first variation of $F_0(R_\alpha z)$
vanishes since $T_0\phi=\mu_1^{-1}\phi$. Thus on $\phi^\perp$,
\[
 F_\alpha(R_\alpha z)=F_0(z)
       +\frac\alpha2\int_\Omega |x|^{-2\beta}z^4\,dx
       +O_{C^2}(\alpha^2).
\]
Differentiating $\nabla_XG_\alpha(h_\alpha)=0$ at zero gives
\eqref{eq:attaining-correction} with $\xi=h'_0$. The gap on $X$
proves existence and uniqueness of $\xi$. Combining this with
$R'_0\eta=-\ip{\psi}{\eta}\phi$ gives \eqref{eq:attaining-expansion}.
\end{proof}

\paragraph{Example: explicit constants on the unit disk.}
Let $\Omega=D$ and $0\le\beta<1$. Denote by $j_{\nu,m}$ the $m$th
positive zero of the Bessel function $J_\nu$. The first two weighted
eigenvalues are
\begin{equation}\label{eq:bessel-eigenvalues}
 \begin{aligned}
 \mu_1&=(1-\beta)^2j_{0,1}^2,\\
 \mu_2&=(1-\beta)^2\min\{j_{0,2}^2,j_{1/(1-\beta),1}^2\}.
 \end{aligned}
\end{equation}
By \cref{thm:sharp}, as $\alpha\downarrow0$,
\[
 C_\alpha
 =\frac{\alpha}{2(1-\beta)^2}
 \left(\frac1{j_{0,1}^2}
       -\frac1{\min\{j_{0,2}^2,j_{1/(1-\beta),1}^2\}}\right)
       +O(\alpha^2).
\]
Either a radial or a nonradial mode can realize the second eigenvalue.

To verify \eqref{eq:bessel-eigenvalues}, set $a=1-\beta$.
For angular frequency $k\ge0$, write
$h(r,\theta)=g(r^a)\cos(k\theta)$,
or use the corresponding sine mode when $k\ge1$.
The substitution $s=r^a$ turns $-\Delta h=\mu r^{-2\beta}h$ into
\[
 -g''-s^{-1}g'+\frac{k^2}{a^2s^2}g=\frac{\mu}{a^2}g.
\]
Finite energy at the origin and the boundary condition at $s=1$
give eigenvalues $a^2j_{k/a,m}^2$. The first radial eigenvalue is
the smallest, and monotonicity of the Rayleigh quotient in $k$
leaves only the second radial and first frequency-one eigenvalues
as candidates for the next, proving \eqref{eq:bessel-eigenvalues}. As $a\downarrow0$ the
angular order $1/a$ grows, so the second radial eigenvalue eventually
becomes the second eigenvalue of the full problem.

The same eigenvalues describe the functions attaining the stability
constant. If $j_{0,2}<j_{1/(1-\beta),1}$, the second eigenspace is
one-dimensional and radial. By \cref{prop:attaining-branch}, for
sufficiently small $\alpha$ there is exactly one attaining pair,
and both functions are radial. Indeed, rotations preserve
$\mathcal A_\alpha=\{w_\alpha,-w_\alpha\}$, so continuity and
connectedness of the rotation group force its orbit to be a single point.

If $j_{1/(1-\beta),1}<j_{0,2}$, the attaining functions approach
the unit sphere of the first angular eigenspace by
\cref{thm:attainment}. All unit vectors in that eigenspace differ
by a rotation, so their weighted fourth-power integrals agree.
At equality, let $f_0,f_c,f_s$ be Dirichlet-orthonormal radial,
cosine, and sine eigenfunctions. After rotation and an overall sign
change, a unit vector has the form
$\sqrt{1-t}\,f_0+\sqrt t\,f_c$, $0\le t\le1$.
The quartic selection reduces to
\[
 \max_{0\le t\le1}\int_D |x|^{-2\beta}
   \bigl((1-t)^2f_0^4+6t(1-t)f_0^2f_c^2+t^2f_c^4\bigr)\,dx.
\]
Angular integration removes the odd mixed terms, so this formula
describes the possible limiting mixtures as $\alpha\downarrow0$.

\section{Critical disk extremals and finite-power stability}\label{sec:disk}

Throughout this section $\Omega=D$, $0\le\beta<1$ and
$0<\alpha\le4\pi(1-\beta)$, with $a=1-\beta$ in radial calculations.
Compactness and second-variation estimates give \cref{thm:disk},
while the radial transformation and \cite[Theorem~C]{LiSun} give
\cref{cor:critical}. We then use the energy of radial solutions to
prove \cref{thm:parameters} and compare constants for different weights.

\subsection{Compactness and stability}

\begin{lemma}\label{lem:transform}
For radial $u\in H_0^1(D)$, define
$U(s)=\sqrt a\,u(s^{1/a})$. Then
\begin{equation}\label{eq:transform-identities}
 \int_D|\nabla U|^2=\int_D|\nabla u|^2,
 \qquad \int_D |x|^{-2\beta}e^{\alpha u^2}\,dx
                 =\frac1a\int_D e^{(\alpha/a) U^2}\,dx.
\end{equation}
The map is an isometry of the radial Dirichlet spaces. Both suprema
are attained within their radial subspaces, and hence
\begin{equation}\label{eq:sup-disk}
 S_\alpha=\frac1a\sup_{\substack{U\in H_0^1(D)\\\norm{U}\le1}}
                    \int_D e^{(\alpha/a)U^2}\,dx.
\end{equation}
Every maximizer has one sign. Positive maximizers are radial and strictly
decreasing. If $u$ is a positive maximizer with multiplier $\lambda$,
then $U$ solves
\begin{equation}\label{eq:Ueq}
 -U''-s^{-1}U'=\frac{\lambda}{a^2}Ue^{(\alpha/a)U^2},\qquad
 U'(0)=0,\quad U(1)=0.
\end{equation}
For radial perturbations the corresponding quadratic forms and tangent
conditions agree exactly under the same isometry.
\end{lemma}
\begin{proof}
The substitution $s=r^a$ gives $s\,ds=a r^{2a-1}\,dr$ and
$U_s=a^{-1/2}r^{1-a}u_r$, proving \eqref{eq:transform-identities}.
Schwarz rearrangement does not increase the Dirichlet norm or decrease
either exponential integral, proving \eqref{eq:sup-disk}.
Existence follows from attainment of the ordinary disk supremum
and rearrangement, for both subcritical exponents and the critical
case \cite{CC}.

Boundedness, continuity and constant sign follow from
\cref{lem:subcompact}, whose estimates near a fixed function also
hold at criticality. For $\beta>0$, the strictly decreasing weight
$r^{-2\beta}$ has its largest integral over a set of given measure
uniquely on the centered disk, so the layer-cake formula for
$e^{\alpha u^2}-1$ forces centered disk superlevel sets in the
equality case. For $\beta=0$, positivity and the moving-plane
theorem \cite{GNN} give radial symmetry, and in both cases the
radial equation gives strict decrease.

Substitution into that equation yields \eqref{eq:Ueq}. If
$g(s)=\sqrt a\,h(s^{1/a})$, then
\[
 \lambda\int_D r^{-2\beta}(1+2\alpha u^2)e^{\alpha u^2}h^2\,dx
  =\frac{\lambda}{a^2}\int_D(1+2(\alpha/a) U^2)e^{(\alpha/a) U^2}g^2\,dx.
\]
Together with the isometry, this proves the last assertion.
\end{proof}

\begin{lemma}\label{lem:criticalcompact}
For $\alpha=4\pi a$, every maximizing sequence in the closed unit ball
of $H_0^1(D)$ has a strongly convergent subsequence.
\end{lemma}
\begin{proof}
Let $u_j\rightharpoonup u$ and $\int_D |x|^{-2\beta}e^{\alpha u_j^2}\,dx\to S_\alpha$.
Suppose first that $u=0$. Let $u_j^*$ be the Schwarz rearrangement of
$|u_j|$. Since $\norm{u_j^*}_{L^2}=\norm{u_j}_{L^2}\to0$ and
$\norm{u_j^*}\le1$, we have $u_j^*\rightharpoonup0$ in $H_0^1(D)$.
The transformed radial functions $U_j$ in \cref{lem:transform} likewise
converge weakly to zero, because that transformation is a bounded
linear isometry on the radial space. We claim that
\begin{equation}\label{eq:CC-level}
 \limsup_j\int_D e^{4\pi U_j^2}\,dx\le\pi(1+e)
                         <aS_{4\pi a}.
\end{equation}
To verify the first inequality, take a subsequence realizing the
limsup and then a further subsequence such that
$|\nabla U_j|^2\,dx\rightharpoonup\nu$ as measures on $\overline D$.
The measure $\nu$ has mass at most one, and if $\nu=\delta_0$,
the radial concentration estimate of Carleson--Chang \cite{CC}
gives the upper bound. To apply it to norms bounded by one, normalize
the sequence, whose norms converge to one in this case.

If $\nu\ne\delta_0$, then $\nu(\{0\})<1$. Choose $0<\rho<1$
with $\nu(\partial B_\rho)=0$ and $\nu(B_\rho)<1$. There is
$\delta>0$ such that, for all sufficiently large $j$,
\[
 \int_{B_\rho}|\nabla U_j|^2\,dx\le1-\delta.
\]
On each closed annulus away from the origin, radial compactness in
one dimension gives uniform convergence $U_j\to0$. Since $U_j$ is
nonnegative and nonincreasing, the function
\[
 W_j(x)=U_j(|x|)-U_j(\rho),\qquad x\in B_\rho,
\]
belongs to $H_0^1(B_\rho)$, is nonnegative, and has squared Dirichlet
norm at most $1-\delta$. Choose $q>1$ and $\varepsilon>0$ so that
$q(1+\varepsilon)(1-\delta)<1$. The inequality
\[
 U_j^2\le(1+\varepsilon)W_j^2
              +(1+\varepsilon^{-1})U_j(\rho)^2
 \quad\text{on }B_\rho
\]
and the ordinary Moser-Trudinger inequality imply
\[
 \sup_j\int_{B_\rho}e^{4\pi qU_j^2}\,dx<\infty.
\]
On $D\setminus B_\rho$ we have $0\le U_j\le U_j(\rho)\to0$.
Vitali's theorem therefore yields
$\int_D e^{4\pi U_j^2}\,dx\to\pi$, which also proves the upper bound
in \eqref{eq:CC-level}. Finally, the strict test-function inequality
of Carleson--Chang \cite{CC}, together with \eqref{eq:sup-disk}, gives
the strict inequality in \eqref{eq:CC-level}.

Rearrangement and \eqref{eq:transform-identities} contradict
maximality, so $u\ne0$.

Choose $b\in C_c^\infty(D)$ sufficiently close to $u$ in $H_0^1(D)$. Weak
convergence implies
\[
 \limsup_j\norm{u_j-b}^2
 \le1-\norm{u}^2+\norm{u-b}^2<1.
\]
Using $(u_j-b+b)^2\le(1+\varepsilon)(u_j-b)^2+
(1+\varepsilon^{-1})b^2$, choose $q>1$ and $\varepsilon>0$ such that
$q(1+\varepsilon)\limsup_j\norm{u_j-b}^2<1$. The singular
Moser-Trudinger inequality then gives
\[
 \sup_j\int_D |x|^{-2\beta} e^{q\alpha u_j^2}\,dx<\infty.
\]
Vitali's theorem gives convergence of the weighted exponential
integrals, making the limit a maximizer of norm one and hence
giving strong convergence.
\end{proof}

We next combine local estimates with compactness to prove global stability.

\begin{lemma}\label{lem:chart}
Let $\Omega\subset\R^2$ be a bounded connected $C^2$ domain containing
$0$. Fix $0\le\beta<1$, and let $v$ be a maximizer for
$0<\alpha\le4\pi(1-\beta)$.
For $h\perp v$ small in $H_0^1(\Omega)$, set
$u=(v+h)/(1+\norm{h}^2)^{1/2}$. This parametrization lies on the unit
sphere and satisfies $\norm{u-v}^2=\norm{h}^2+O(\norm{h}^4)$.
Moreover,
\begin{equation}\label{eq:chart-exp}
 S_\alpha-\int_\Omega |x|^{-2\beta}e^{\alpha u^2}\,dx
       =\frac{\alpha}{\lambda_v}Q_v[h]+O(\norm{h}^3).
\end{equation}
The functional is real analytic in a sufficiently small $H_0^1(\Omega)$-neighborhood
of $v$. The operator representing $Q_v$ on $v^\perp$ is the identity
minus a compact operator. Also $Q_v\ge0$ on $v^\perp$. If
$Q_v[h]=0$ with $h\perp v$ implies $h=0$, then $Q_v$ is coercive.
\end{lemma}
\begin{proof}
Boundedness of $v$ and the singular Moser-Trudinger inequality,
applied with $(v+h)^2\le2\norm{v}_\infty^2+2h^2$, control the
integral power series for small $\norm{h}$, including polynomial
factors and derivatives after reducing the radius.
The same bound with $|h|$ applies to complex directions, so the
complexified integral is locally bounded and holomorphic on each
affine complex line, hence Fr\'echet holomorphic and real analytic
on the real space. Since $h\perp v$, the linear term vanishes by
\eqref{eq:EL}, and expanding the exponential and normalization gives
\begin{align*}
 \int_\Omega |x|^{-2\beta}e^{\alpha u^2}\,dx
 &=S_\alpha-\frac{\alpha}{\lambda_v}\norm{h}^2\\
 &\quad+\alpha\int_\Omega |x|^{-2\beta}
             (1+2\alpha v^2)e^{\alpha v^2}h^2\,dx+O(\norm{h}^3).
\end{align*}
This proves \eqref{eq:chart-exp}, while boundedness of $v$ and
\cref{lem:weighted} give compactness of the potential term.
Maximality gives nonnegativity, so the spectral theorem yields
a positive lower bound when the kernel is trivial.
\end{proof}

\begin{proposition}
\label{prop:nonradial}
For each positive maximizer $v$ on $D$ and every
$0<\alpha\le4\pi(1-\beta)$, there exists $c>0$ such that
\begin{equation}\label{eq:nonradialcoercivity}
 Q_v[h]\ge c\norm{h}^2
\end{equation}
whenever $h$ has zero angular average. Such functions belong to $v^\perp$.
Thus every $h\perp v$ with $Q_v[h]=0$ is radial.
\end{proposition}
\begin{proof}
Use \cref{lem:transform}, and write
$V(s)=(1+2(\alpha/a) U(s)^2)e^{(\alpha/a) U(s)^2}$. For a cosine or sine mode of
angular frequency $k\ge1$, set
$h(r,\theta)=a^{-1/2}g(r^a)\cos(k\theta)$, or the corresponding sine.
A direct calculation gives
\begin{align}
 \norm{h}^2&=\pi\int_0^1
       \left(s|g'|^2+\frac{k^2}{a^2s}g^2\right)ds,\label{eq:mode-energy}\\
 Q_v[h]&=\pi\int_0^1
       \left(s|g'|^2+\frac{k^2}{a^2s}g^2-\frac{\lambda_v}{a^2} sVg^2\right)ds.
       \label{eq:mode-Q}
\end{align}
Let $\rho_1$ be the first Dirichlet eigenvalue of
\begin{equation}\label{eq:mode1}
 -(s\zeta')'+s^{-1}\zeta=\rho sV\zeta,
 \qquad \zeta(1)=0,
\end{equation}
with the finite-energy condition at $0$. Its first eigenfunction
$\zeta$ is positive in $(0,1)$ and $\zeta'(1)<0$.
The function $\chi=-U'>0$ solves the same differential equation with
$\rho=\frac{\lambda_v}{a^2}$, but $\chi(1)>0$. Near $0$, both functions are $O(s)$.
Green's identity therefore gives
\begin{equation}\label{eq:wronskian}
 (\rho_1-\frac{\lambda_v}{a^2})\int_0^1sV\chi\zeta\,ds
                    =-\chi(1)\zeta'(1)>0.
\end{equation}
Thus $\rho_1>\frac{\lambda_v}{a^2}$, and $k/a\ge1$ bounds the
Rayleigh quotient in \eqref{eq:mode-energy} below by $\rho_1$, giving
\[
 Q_v[h]\ge\left(1-\frac{\lambda_v}{a^2\rho_1}\right)\norm{h}^2
\]
on every nonradial mode, and summing the orthogonal Fourier modes
proves \eqref{eq:nonradialcoercivity}. Since the radial and nonradial
parts are orthogonal for both $Q_v$ and the Dirichlet form,
$Q_v[h]=0$ with $h\perp v$ forces the nonradial part to vanish.
\end{proof}

\begin{remark}\label{rem:constrained-coercivity}
Coercivity on the tangent space $v^\perp$ follows by combining the
nonradial estimate in \cref{prop:nonradial}, with coefficient
$1-\lambda_v/(a^2\rho_1)$, and radial constrained nondegeneracy
\eqref{eq:RN}, characterized by \cref{prop:shooting}.
At the critical exponent, the radial bound follows from
\cite[Theorem~C and Proposition~7.5]{LiSun} through the change of
variables in \cref{lem:transform}.
\end{remark}

When the second variation vanishes in some directions, we reduce the
estimate to those directions and apply the finite-dimensional
\L{}ojasiewicz inequality \cite[Section~6]{BM}.

\begin{lemma}\label{lem:analytic-stability}
Let $f$ be real analytic near $0$ in a Hilbert space, with
$f\ge0$, $f(0)=0$, $Df(0)=0$. Suppose $D^2f(0)$ is nonnegative,
has a finite-dimensional kernel, and is coercive on its orthogonal
complement. Let $Z$ consist of the points in a sufficiently small
neighborhood of $0$ where $f=0$. Then there exist $C>0$ and finite
$p\ge2$ such that $f(x)\ge C\dist(x,Z)^p$ for all $x$ sufficiently close to $0$.
If the kernel is trivial, one may take $p=2$.
\end{lemma}
\begin{proof}
Write $x=k+z$ with $k$ in the kernel $K$ and $z\in K^\perp$.
The analytic implicit function theorem solves
$P_{K^\perp}Df(k+z)=0$ as $z=\Psi(k)$ near $0$.
The Hessian remains uniformly positive in the $K^\perp$ directions, giving
\begin{equation}\label{eq:LS-lower}
 f(k+z)\ge g(k)+c\norm{z-\Psi(k)}^2,
 \qquad g(k)=f(k+\Psi(k))\ge0.
\end{equation}
Since $g\ge0$, \eqref{eq:LS-lower} and the definition of $g$ show that
$f(k+z)=0$ exactly when $z=\Psi(k)$ and $g(k)=0$.
The finite-dimensional analytic distance
inequality gives $g(k)\ge c\dist(k,\{g=0\})^p$ for some finite
$p\ge2$, after increasing the exponent if needed. If $g$ vanishes
identically, the distance is zero and \eqref{eq:LS-lower} suffices.
On a smaller neighborhood, $\Psi$ is Lipschitz, so
\[
 \dist(k+z,Z)\le C\bigl(\dist(k,\{g=0\})
                            +\norm{z-\Psi(k)}\bigr).
\]
For $p\ge2$ and small $z-\Psi(k)$, the squared norm bounds its
$p$th power, so these inequalities and \eqref{eq:LS-lower} give
the estimate. When $K=\{0\}$, Taylor's theorem gives it with $p=2$.
\end{proof}

\begin{proof}[Proof of the stability assertions in \cref{thm:disk}]
Compactness follows from \cref{lem:subcompact,lem:criticalcompact}.
Near each maximizer, \cref{lem:chart} expresses the deficit as an
analytic function vanishing exactly at maximizers, so
\cref{lem:analytic-stability} applies. Distances in these coordinates
are comparable to Dirichlet distances, and allowing all of
$\M_\alpha$ only decreases the distance.

Compactness of $\M_\alpha$ gives a finite cover by smaller
neighborhoods with local estimates, which share one bound after
taking the largest exponent and reducing the constants.
Outside this cover the deficit has a positive lower bound, since
otherwise a maximizing sequence would stay away from $\M_\alpha$,
contradicting compactness. The distance is bounded on the sphere,
so the finite-power estimate holds globally there.

If \eqref{eq:RN} holds, split each tangent vector into radial
and zero-angular-average parts, both tangent and orthogonal for
the Dirichlet form and $Q_v$. Condition \eqref{eq:RN} and
\cref{prop:nonradial} then give a trivial constrained kernel at
every positive maximizer, and also at negative maximizers because
$Q_{-v}=Q_v$ and $(-v)^\perp=v^\perp$.
The preceding argument therefore applies with exponent $2$.

To pass to the closed ball, write $u=tz$, $\norm{z}=1$, as in
\cref{lem:ball}. Convexity and the sphere estimate give
\begin{equation}\label{eq:radial-convex}
 \begin{aligned}
 S_\alpha-\int_D |x|^{-2\beta}e^{\alpha t^2z^2}\,dx
 &\ge Ct^2\dist(z,\M_\alpha)^p\\
 &\quad+(1-t^2)\left(S_\alpha-\int_D |x|^{-2\beta}\,dx\right).
 \end{aligned}
\end{equation}
For $t\ge1/2$, the right side controls
$\dist(z,\M_\alpha)^p+(1-t)^p$ and hence
$\dist(tz,\M_\alpha)^p$ by the triangle inequality.
For $t<1/2$, including $u=0$, the deficit has a fixed positive
lower bound and the distance is bounded, proving
\eqref{eq:finitepower} and its quadratic version under \eqref{eq:RN}.
\end{proof}

\begin{proof}[Proof of \cref{cor:critical}]
By \cite[Theorem~C]{LiSun}, the ordinary critical maximizers are
exactly $\pm u_*$, with a coercive constrained Hessian at $u_*$.
Every positive weighted maximizer is radial by \cref{lem:transform}
and maps to $u_*$ under \eqref{eq:radial-transform}. Conversely,
$v_\beta(r)=a^{-1/2}u_*(r^a)$ attains the weighted supremum,
proving \eqref{eq:critical-profile}.

If $h\perp v_\beta$ is radial, set
$g(s)=\sqrt a\,h(s^{1/a})$. By \cref{lem:transform}, we have
$g\perp u_*$, $\norm{g}=\norm{h}$, and the multiplier of $u_*$ is
$\lambda_{v_\beta}/a^2$. Thus \cite[Proposition~7.5]{LiSun} gives
\[
 \begin{aligned}
 Q_{v_\beta}[h]
 &=\norm{g}^2-\frac{\lambda_{v_\beta}}{a^2}
       \int_D(1+8\pi u_*^2)e^{4\pi u_*^2}g^2\,dx\\
 &\ge c\norm{g}^2=c\norm{h}^2
 \end{aligned}
\]
for some $c>0$. For an arbitrary tangent vector let $h_0$ be its
angular average. The functions $h_0$ and $h-h_0$ are tangent and
orthogonal for both the Dirichlet form and $Q_{v_\beta}$.
Combining the radial bound with \cref{prop:nonradial}, and decreasing
$c$ if necessary, yields
\[
 Q_{v_\beta}[h]
 =Q_{v_\beta}[h_0]+Q_{v_\beta}[h-h_0]
 \ge c\bigl(\norm{h_0}^2+\norm{h-h_0}^2\bigr)=c\norm{h}^2.
\]
The resulting full constrained coercivity and \cref{thm:disk}
with $p=2$ give \eqref{eq:critical-stability} on the closed unit ball.
For any nonzero $h\perp v_\beta$, the curve
$u_t=(v_\beta+th)/(1+t^2\norm{h}^2)^{1/2}$ has deficit $O(t^2)$
and $\dist(u_t,\M_{4\pi a})\sim |t|\norm{h}$ by \cref{lem:chart},
which proves sharpness of the exponent two.
\end{proof}

\begin{corollary}\label{cor:general}
Let $\Omega\subset\R^2$ be a bounded connected $C^2$ domain containing
$0$, and fix $0\le\beta<1$. For each $0<\alpha<4\pi(1-\beta)$,
every maximizing sequence has a strongly convergent subsequence,
and there exist $C>0$ and finite $p\ge2$ such that
\eqref{eq:finitepower} holds. If $\partial\Omega$ is smooth, both
conclusions also hold at $\alpha=4\pi(1-\beta)$.
If $Q_v$ has trivial kernel on $v^\perp$ at every maximizer,
one may take $p=2$.
\end{corollary}
\begin{proof}
Subcritical compactness is \cref{lem:subcompact}. At the critical
exponent, subtract $\int_\Omega |x|^{-2\beta}\,dx$ from the functional.
For $\beta=0$, Flucher's concentration estimate and strict
test-function inequality \cite{Flucher} place the maximal
concentration level strictly below the supremum.
For $0<\beta<1$, Csat\'o--Roy \cite[Proposition~7]{CR}, with
their singular exponent equal to $2\beta$, give concentration
level zero away from the origin. Their disk gap, transplantation
bound, and concentration formula
\cite[Theorems~15, 16 and 21]{CR} place the concentration level
at the origin strictly below the supremum as well.
The concentration-compactness alternative
\cite[Theorem~1]{Flucher}, \cite[Theorem~6]{CR} therefore gives
a weak limit at which the functional converges to its supremum.
The limit is nonzero and has norm one by amplitude rescaling,
so convergence is strong, proving existence and compactness of
the maximizer set.

The exponential estimate and elliptic regularity in
\cref{lem:subcompact} give boundedness of every maximizer, including
at the endpoint. Applying \cref{lem:chart,lem:analytic-stability}
near each maximizer and taking a finite cover gives the sphere
estimate, while compactness gives a positive deficit elsewhere.
The convexity argument in \eqref{eq:radial-convex}, with $\Omega$
in place of $D$, extends the estimate to the closed ball, with
exponent two when the constrained kernels are trivial.
\end{proof}

\flushbottom
\begingroup
\setlength{\abovedisplayskip}{6pt plus 2pt minus 2pt}
\setlength{\belowdisplayskip}{6pt plus 2pt minus 2pt}
\setlength{\abovedisplayshortskip}{0pt plus 2pt}
\setlength{\belowdisplayshortskip}{4pt plus 2pt minus 1pt}
\subsection{Nondegeneracy and parameter dependence}

The condition $Q_v[h]=0$ on $v^\perp$ amounts to linearizing
the equation while allowing the multiplier to vary: since
$Q_v\ge0$ on $v^\perp$, it is equivalent to the existence of
$c\in\R$ such that
\begin{equation}\label{eq:bordered}
 \begin{cases}
 -\Delta h=|x|^{-2\beta}e^{\alpha v^2}
                  \bigl[\lambda_v(1+2\alpha v^2)h+cv\bigr]&\text{in }D,\\
 h=0&\text{on }\partial D,\\
 \displaystyle\int_D\nabla v\cdot\nabla h\,dx=0.
 \end{cases}
\end{equation}
Indeed, nonnegativity makes the associated bilinear form vanish
on $(h,k)$ for every $k\in v^\perp$, so its value on arbitrary
$k$ is a scalar multiple of $\ip{v}{k}$.
By \eqref{eq:EL}, this equals a multiple of
$\lambda_v\int_D |x|^{-2\beta}ve^{\alpha v^2}k\,dx$,
giving \eqref{eq:bordered} after absorbing the factor into $c$.
Conversely, testing the first equation against $h\perp v$ gives
$Q_v[h]=0$.

For each $b>0$, let $y$ solve the initial-value problem
\begin{equation}\label{eq:shoot}
 y''+t^{-1}y'+ye^{y^2}=0,
 \qquad y(0)=b,\quad y'(0)=0.
\end{equation}
Let $R(b)>0$ be the first point where $y$ vanishes. Rescale
$[0,R(b)]$ to $[0,1]$ by setting
\begin{equation}\label{eq:shoot-energy}
 z(s)=y(R(b)s),\qquad
 \mathcal E(b)=2\pi\int_0^1s|z'(s)|^2\,ds.
\end{equation}
Then $z(0)=b$ and $z(1)=0$, and $\mathcal E(b)$ is the Dirichlet
energy of $z$ on the unit disk. Both $y$ and $z$ depend on $b$,
and $\partial_b$ denotes differentiation with respect to $b$.
While $y$ is positive,
\[
 t y'(t)=-\int_0^t\tau y(\tau)e^{y(\tau)^2}\,d\tau<0.
\]
Sturm comparison with $J_0$ gives a finite first zero, and the
identity above gives $y'(R(b))<0$.
Analytic dependence at $t=0$ follows from the integral equation
\begin{equation}\label{eq:radial-volterra}
 y(t;b)=b-\int_0^t s\log(t/s)y(s;b)e^{y(s;b)^2}\,ds.
\end{equation}
Near $b_0>0$, the right side is an analytic contraction on a
fixed ball in $C([0,T])$ for small $T$, since the kernel has norm
$T^2/4$ and $y\mapsto ye^{y^2}$ is analytic and locally Lipschitz.
The analytic implicit function theorem and the formula for $ty'(t)$
give analytic dependence on $b$ in $C^1([0,T])$, which continuation
extends past $R(b_0)$. The inequality $y'(R(b_0))<0$ makes $R(b)$
locally analytic, and the rescaled integral equation gives the same
dependence for $z$ in $C^1([0,1])$ and for $\mathcal E(b)$.

\begin{proposition}
\label{prop:shooting}
Let $v$ be a positive maximizer for the weighted disk problem,
set $b=\sqrt\alpha\,v(0)$, and define $z$ by \eqref{eq:shoot-energy}.
Then
\begin{equation}\label{eq:shoot-identification}
 z(s)=\sqrt\alpha\,v(s^{1/a}),\qquad
 \mathcal E(b)=\frac{\alpha}{a}.
\end{equation}
The radial constrained Hessian at $v$ has trivial kernel if and only if
\begin{equation}\label{eq:shoot-criterion}
 \mathcal E'(b)\ne0.
\end{equation}
\end{proposition}
\begin{proof}
By \cref{lem:transform}, $\sqrt\alpha\,v(s^{1/a})$ solves
$-\Delta z=(\lambda_v/a^2)ze^{z^2}$ with energy $\alpha/a$
and central value $b$. Uniqueness for the radial initial-value
problem and its first zero identify this function with $z$
and give $\lambda_v/a^2=R(b)^2$, proving \eqref{eq:shoot-identification}.

Set $L=-\Delta-R(b)^2(1+2z^2)e^{z^2}$ on radial functions.
Differentiation with respect to $b$ gives
\begin{equation}\label{eq:shoot-derivative}
 L\partial_b z=2R(b)R'(b)ze^{z^2},\qquad
 \partial_b z(0)=1,\quad \partial_b z(1)=0,
\end{equation}
and
\begin{equation}\label{eq:energy-derivative}
 \mathcal E'(b)=2\int_D\nabla z\cdot\nabla\partial_b z\,dx.
\end{equation}
For any radial pair $(h,c)$ with $Lh=cze^{z^2}$ and $h(1)=0$,
elliptic regularity gives a well-defined $h(0)$ and $h'(0)=0$.
Subtracting $h(0)$ times the pair in \eqref{eq:shoot-derivative}
leaves $\widetilde h(0)=0$ and
$L\widetilde h=\widetilde c\,ze^{z^2}$, while differentiating the
equation for $z$ gives
\[
 L(sz')=2R(b)^2ze^{z^2}.
\]
Both $sz'$ and its derivative vanish at the origin. Uniqueness for
the regular radial initial-value problem therefore yields
\[
 \widetilde h(s)=\frac{\widetilde c}{2R(b)^2}s z'(s).
\]
Since $\widetilde h(1)=0$ and $z'(1)<0$, we have
$\widetilde c=0$ and $\widetilde h=0$, so all radial pairs are
multiples of $(\partial_b z,2R(b)R'(b))$, including when $R'(b)=0$.
By \eqref{eq:energy-derivative}, a nonzero such solution satisfies
the orthogonality condition exactly when $\mathcal E'(b)=0$,
which proves the criterion in the original variables.
\end{proof}

\begin{remark}
Chen--Lu--Xue--Zhu \cite{CLXZ} study uniqueness at a prescribed
multiplier. Here the constraint prescribes the Dirichlet energy,
as in Malchiodi--Martinazzi \cite{MM}. By \eqref{eq:shoot-criterion},
radial nondegeneracy is equivalent to $\mathcal E'(b)\ne0$.
\end{remark}

\begin{remark}
At $\alpha=4\pi(1-\beta)$, the energy level is $\mathcal E(b)=4\pi$.
By \eqref{eq:critical-profile}, the critical maximizer corresponds to
$b_*=\sqrt{4\pi}\,u_*(0)$, independent of $\beta$.
The initial value and normalized energy in
\cite[Section~7.3]{LiSun} correspond to $b$ and
$\mathcal E(b)/(4\pi)$ here. Their Corollary~7.3 gives
$\mathcal E'(b_*)>0$, which yields the radial nondegeneracy in
\cref{cor:critical} by \cref{prop:shooting}.
\end{remark}
\endgroup

We now let the exponent vary. Write $J_{\tau,0}(U)=\int_D e^{\tau U^2}\,dx$,
$S_{\tau,0}$, and $\M_{\tau,0}$ for the unweighted disk functional,
its supremum, and its maximizer set. By \cref{lem:transform},
positive weighted maximizers correspond to the positive elements of
$\M_{\alpha/a,0}$.

\begin{lemma}\label{lem:parameter-compact}
Suppose that $0<\tau_j\le4\pi$, $\tau_j\to\tau\in(0,4\pi]$,
$\norm{U_j}\le1$, and
$S_{\tau_j,0}-J_{\tau_j,0}(U_j)\to0$. Then a subsequence converges
strongly in $H_0^1(D)$ to an element of $\M_{\tau,0}$.
If each $U_j$ is a positive maximizer, convergence also holds in
$C^1(\overline D)$ and the multipliers converge.
Consider all positive maximizers for $0<\tau\le4\pi$, and include
the normalized positive first Dirichlet eigenfunction at $\tau=0$.
The resulting set of pairs $(\tau,U)$ is compact in
$[0,4\pi]\times C^1(\overline D)$.
\end{lemma}
\begin{proof}
For $\tau<4\pi$, choose $q>1$ and $\eta>0$ with
$q(\tau+\eta)<4\pi$ to obtain uniform integrability on the
unit ball as in \cref{lem:subcompact}.
After weak and almost everywhere convergence, Vitali's theorem gives
$J_{\tau_j,0}(U_j)\to J_{\tau,0}(U)$. Testing with a fixed maximizer at
$\tau$ gives $\liminf S_{\tau_j,0}\ge S_{\tau,0}$. Conversely, applying
the same weak compactness argument to maximizers for $S_{\tau_j,0}$
gives $\limsup S_{\tau_j,0}\le S_{\tau,0}$. Hence $U$ is a maximizer,
has norm one, and the convergence is strong.

If $\tau=4\pi$, monotonicity and a fixed critical maximizer $U_*$
give
\[
 J_{\tau_j,0}(U_*)\le S_{\tau_j,0}\le S_{4\pi,0},\qquad
 J_{\tau_j,0}(U_*)\longrightarrow S_{4\pi,0}.
\]
Since $J_{4\pi,0}(U_j)\ge J_{\tau_j,0}(U_j)\to S_{4\pi,0}$,
the sequence also maximizes the fixed critical problem, so
\cref{lem:criticalcompact} with $\beta=0$ gives strong convergence.

For maximizers the multiplier is
$\kappa_j=(\int_D U_j^2e^{\tau_jU_j^2}\,dx)^{-1}$.
Strong $H_0^1$ convergence bounds $e^{B U_j^2}$ uniformly in
every fixed finite $L^p$ space for each $B>0$, by approximating
the limit with a bounded smooth function and making the remainder
small as in \cref{lem:subcompact}.
H\"older's inequality and Vitali's theorem then give convergence of
the multiplier integrals and of $\kappa_j U_j e^{\tau_j U_j^2}$
in every finite $L^p(D)$, so elliptic estimates give convergence
in $W^{2,p}(D)$ and hence in $C^1(\overline D)$ for $p>2$.

For $\tau_j\downarrow0$, \cref{lem:limit} with $\beta=0$
gives strong convergence to the first eigenfunction and convergence
of the multipliers, which the same elliptic argument improves to
$C^1$ convergence. Taking a convergent parameter subsequence proves
compactness after including this eigenfunction at $\tau=0$.
\end{proof}

\begin{lemma}\label{lem:small-shoot}
The energy function $\mathcal E$ in \eqref{eq:shoot-energy} satisfies
\begin{equation}\label{eq:small-shoot}
 \mathcal E(b)=c_0 b^2+O(b^4),\qquad
 \mathcal E'(b)=2c_0 b+O(b^3),\qquad c_0>0,
\end{equation}
as $b\downarrow0$. In particular $\mathcal E'$ is positive near
zero and is not identically zero on $(0,\infty)$.
\end{lemma}
\begin{proof}
Set $q=y/b$. Its equation is
$q''+t^{-1}q'+q e^{b^2q^2}=0$ with
$q(0)=1$ and $q'(0)=0$. The integral form
\[
 q(t)=1-\int_0^t s\log(t/s)q(s)e^{b^2q(s)^2}\,ds
\]
gives analytic dependence on $b^2$, first locally by a contraction
argument and then on compact intervals by continuation.
At $b=0$, $q=J_0$, whose first positive zero $j_{0,1}$ is simple.
Thus $R(b)=j_{0,1}+O(b^2)$ and
$z(s)=b J_0(j_{0,1}s)+O(b^3)$ in $C^1([0,1])$.
Substitution into the energy gives \eqref{eq:small-shoot}, with
\[
 c_0=2\pi\int_0^1 s\left|\frac{d}{ds}J_0(j_{0,1}s)\right|^2ds>0.
\]
The maps $R(b)$ and $z/b$ are analytic in $b^2$, which also justifies
differentiating the remainder.
\end{proof}

\begin{proof}[Proof of \cref{thm:parameters}]
By \cref{lem:parameter-compact}, the central values
$b=\sqrt\tau\,U(0)$ associated with positive unweighted maximizers satisfy
$b\le M$ for a common $M<\infty$ and tend to zero as $\tau\downarrow0$.
By \cref{lem:small-shoot}, choose $d\in(0,M)$ with
$\mathcal E'>0$ on $(0,d)$, while analyticity and
$\mathcal E'\not\equiv0$ give finitely many zeros in $[d,M]$.
Define $\Sigma$ as the set of parameters $\tau\in(0,4\pi]$ for
which some positive maximizer satisfies
\[
 \mathcal E'(\sqrt\tau\,U(0))=0.
\]
The identity $\mathcal E(\sqrt\tau\,U(0))=\tau$ makes $\Sigma$
finite, and \cref{cor:critical,prop:shooting} exclude $4\pi$,
so $\Sigma\subset(0,4\pi)$.

For fixed $\tau>0$, the central values lie in a compact subset
of $(0,M]$, where $\mathcal E(b)-\tau$ has finitely many zeros
by analyticity and nonconstancy. Each central value determines the
positive radial solution, so $\M_{\tau,0}$ is finite.
The isometry in \cref{lem:transform} preserves radial constrained
forms, and \cref{prop:nonradial} excludes nonradial zero modes,
so the same $\Sigma$ works for every $\beta$.
At a degenerate maximizer, \cref{prop:shooting} makes every radial
linearized solution a multiple of $\partial_b z$, which is tangent
because $\mathcal E'(b)=0$, giving a one-dimensional full
constrained kernel.

Outside $a\Sigma$, \cref{thm:disk} gives quadratic stability.
For $\alpha\in a\Sigma$, choose a degenerate maximizer $v$ and
$0\ne h\in v^\perp$ in its kernel. Along $u_t=(v+th)/(1+t^2\norm{h}^2)^{1/2}$, the analytic
deficit has zero terms through order two and is nonnegative for
both signs of $t$, so its cubic term also vanishes and it is $O(t^4)$.
Finiteness of $\M_\alpha$ gives
$\dist(u_t,\M_\alpha)\sim |t|\norm{h}$, ruling out every exponent
$p<4$. This proves \textup{(i)} and \textup{(ii)}.

For \textup{(iii)}, we use the critical uniqueness and coercivity
in \cref{cor:critical}, obtained from \cite[Theorem~C]{LiSun}.
Consider
\[
 (\tau,U,\kappa)\longmapsto
 \left(U-\kappa(-\Delta)^{-1}(Ue^{\tau U^2}),\ \norm{U}^2-1\right)
\]
near the unique positive critical extremal
$(4\pi,u_*,\kappa_*)\in\R\times H_0^1(D)\times\R$.
The map is analytic by \cref{lem:chart}. Its derivative in
$(U,\kappa)$ is invertible: the constraint fixes the component
parallel to $u_*$, coercivity determines the tangent component,
and projection onto $u_*$ determines the multiplier variation.
The analytic implicit function theorem gives a family $U_\tau$
which, by \cref{lem:parameter-compact}, contains every positive
maximizer for $\tau\in[4\pi-\delta,4\pi]$ with $\delta>0$ small
and independent of $\beta$. This proves uniqueness there and
analyticity of $v_\alpha(r)=a^{-1/2}U_{\alpha/a}(r^a)$ in $H_0^1(D)$.

The $C^1$ continuity of $U_\tau$ gives operator-norm continuity
of the potentials, so orthogonal rotations between tangent spaces
transfer critical coercivity to a uniform radial bound after
decreasing $\delta$. The radial transformation preserves this bound.
For nonradial directions, continuity of $\kappa_\tau$ and the first
eigenvalue $\rho_1(\tau)$ in \eqref{eq:mode1}, together with the
Wronskian identity, gives
$\min_\tau(1-\kappa_\tau/\rho_1(\tau))>0$ on
$[4\pi-\delta,4\pi]$. Since $k/a\ge1$, this controls every
transformed angular mode and proves \eqref{eq:uniform-end-Q}.

For fixed $\beta$, Taylor's formula gives a uniform local
quadratic estimate. If the global estimate failed, some
$\alpha_j\to\bar\alpha$ and unit vectors $w_j$ would have
deficit-to-distance quotients tending to zero, hence vanishing
deficits because the distances are bounded, while continuity of
$v_\alpha$ gives $S_{\alpha_j}\to S_{\bar\alpha}$.
If $\bar\alpha<4\pi a$, choose $q>1$ with $q\alpha_j<4\pi a$
for large $j$ and pass to weak and almost everywhere convergence.
The weighted Vitali argument in \cref{lem:subcompact} then gives
$J_{\alpha_j}(w_j)\to J_{\bar\alpha}(w)$, making $w$ a maximizer
of norm one and giving strong convergence.
If $\bar\alpha=4\pi a$, monotonicity gives
$J_{4\pi a}(w_j)\ge J_{\alpha_j}(w_j)\to S_{4\pi a}$,
so \cref{lem:criticalcompact} gives the same conclusion.
Uniqueness and continuity of $v_\alpha$ then contradict the uniform
local estimate. Lemma~\ref{lem:ball} extends the bound to the closed unit
ball, proving \eqref{eq:uniform-end-stability}.
\end{proof}

\subsection{Comparison of singular weights}

We now display the weight parameter in $J_{\alpha,\beta}$,
$S_{\alpha,\beta}$, $\M_{\alpha,\beta}$ and the sharp constant
$C_{\alpha,\beta}\ge0$ from \eqref{eq:bestconstant}.
Its radial counterpart $C_{\alpha,\beta}^{\mathrm{rad}}$ restricts
the infimum to radial competitors, with distance to the same
maximizer set because all maximizers are radial.

The power transformation below is the radial map in
\cite[Lemma~10]{CR}, extended to arbitrary functions as in
\cite{DLL}. Keeping its angular energy term gives a comparison of
the stability constants.

\begin{proposition}\label{prop:weight-comparison}
For $0\le\beta<1$, set $a=1-\beta$ and $\tau=\alpha/a$.
For every $0<\alpha\le4\pi a$,
\begin{equation}\label{eq:weight-comparison}
 \frac{C_{\tau,0}}a\le C_{\alpha,\beta}
 \le C_{\alpha,\beta}^{\mathrm{rad}}
 =\frac{C_{\tau,0}^{\mathrm{rad}}}a.
\end{equation}
For fixed $0<\tau\le4\pi$ and $0\le\beta_1\le\beta_2<1$,
\begin{equation}\label{eq:weight-monotone}
 (1-\beta_1)C_{(1-\beta_1)\tau,\beta_1}
 \le(1-\beta_2)C_{(1-\beta_2)\tau,\beta_2}.
\end{equation}
Throughout the admissible range,
\begin{equation}\label{eq:weight-upper}
 C_{\alpha,\beta}\le K\frac{\alpha}{(1-\beta)^2},
 \qquad K=\frac{S_{4\pi,0}-\pi}{8\pi}.
\end{equation}
There are $\varepsilon,\delta,c>0$, independent of $\alpha,\beta$,
such that
\begin{equation}\label{eq:weight-scale}
 c\frac{\alpha}{(1-\beta)^2}\le C_{\alpha,\beta}
       \le K\frac{\alpha}{(1-\beta)^2}
\end{equation}
whenever $0<\tau\le\varepsilon$ or $4\pi-\delta\le\tau\le4\pi$.
\end{proposition}
\begin{proof}
For $u\in H_0^1(D)$, set
$U(s,\theta)=\sqrt a\,u(s^{1/a},\theta)$ and define
\[
 E_a(U)=\int_0^{2\pi}\int_0^1
          \left(|U_s|^2+\frac{|U_\theta|^2}{a^2s^2}\right)s\,ds\,d\theta.
\]
The change of variables $s=r^a$ gives
\begin{equation}\label{eq:angular-transform}
 \begin{aligned}
 E_a(U)&=\norm{u}^2,\qquad
       \norm{U}^2\le E_a(U)\le a^{-2}\norm{U}^2,\\
 J_{\alpha,\beta}(u)&=a^{-1}J_{\tau,0}(U),\qquad
 S_{\alpha,\beta}=a^{-1}S_{\tau,0}.
 \end{aligned}
\end{equation}
Direct differentiation gives the energy identities for smooth
functions, and density and equivalent norms extend the map and
its inverse to $H_0^1(D)$. Almost-everywhere convergence identifies
the extension with the stated substitution, so change of variables
gives the integral identity, while radial maximizers and the
isometry in \cref{lem:transform} give the supremum identity.

Under this map, $\M_{\alpha,\beta}$ corresponds exactly to
$\M_{\tau,0}$, whose elements $V$ are radial and have norm one,
so the cross term with $V$ is independent of the angular
coefficient in $E_a$ and
\begin{equation}\label{eq:distance-transform}
 \dist(u,\M_{\alpha,\beta})^2
 =\dist(U,\M_{\tau,0})^2+E_a(U)-\norm{U}^2.
\end{equation}
For $E_a(U)\le1$, Lemma~\ref{lem:ball} gives
\begin{align*}
 S_{\alpha,\beta}-J_{\alpha,\beta}(u)
 &=a^{-1}\bigl(S_{\tau,0}-J_{\tau,0}(U)\bigr)\\
 &\ge a^{-1}C_{\tau,0}
       \bigl(\dist(U,\M_{\tau,0})^2+1-\norm{U}^2\bigr)\\
 &\ge a^{-1}C_{\tau,0}\dist(u,\M_{\alpha,\beta})^2.
\end{align*}
The compactness needed here follows from \cref{thm:disk}, so this
proves the lower comparison. Restricting competitors
to radial functions gives the upper comparison and equality in
\eqref{eq:weight-comparison}.

For monotonicity, let $a_i=1-\beta_i$, so $a_2\le a_1$ and
$E_{a_1}\le E_{a_2}$. In both Hilbert norms the supremum of
$J_{\tau,0}$ is $S_{\tau,0}$ and its maximizer set is $\M_{\tau,0}$:
the sets $\{U:E_{a_i}(U)\le1\}$ lie in $\{U:\norm{U}\le1\}$
and contain all its radial maximizers. Write
$d_i(U)^2=\min_{V\in\M_{\tau,0}}E_{a_i}(U-V)$.
The sharp constant in the $E_{a_i}$ norm is
$a_iC_{a_i\tau,\beta_i}$ by \eqref{eq:angular-transform}.
For $E_{a_2}(U)\le1$, the common cross terms give
\[
 d_2(U)^2=d_1(U)^2+E_{a_2}(U)-E_{a_1}(U)
          \le d_1(U)^2+1-E_{a_1}(U).
\]
Lemma~\ref{lem:ball} in the $E_{a_1}$ norm makes
$a_1C_{a_1\tau,\beta_1}$ an admissible constant in the
$E_{a_2}$ norm, proving \eqref{eq:weight-monotone}.

For the upper bound, choose a radial orthonormal sequence,
whose squared distances to $\M_{\tau,0}$ tend to two by
compactness. Since $J_{\tau,0}\ge\pi$, this gives
$C_{\tau,0}^{\mathrm{rad}}\le(S_{\tau,0}-\pi)/2$.
Convexity of the exponential gives
\[
 S_{\tau,0}-\pi\le\frac{\tau}{4\pi}(S_{4\pi,0}-\pi).
\]
Together with \eqref{eq:weight-comparison}, this proves
\eqref{eq:weight-upper}. At $\beta=0$, \cref{thm:sharp} for
small $\tau$ and \cref{thm:parameters} on $[4\pi-\delta,4\pi]$
give $C_{\tau,0}\ge c\tau$ after reducing $c$, so
\eqref{eq:weight-comparison} yields \eqref{eq:weight-scale}.
\end{proof}

When $j_{0,2}<j_{1/(1-\beta),1}$, the radial attaining pair found
in Section~\ref{sec:stability} gives, for sufficiently small $\alpha$,
\[
 C_{\alpha,\beta}=C_{\alpha,\beta}^{\mathrm{rad}}
      =\frac{C_{\alpha/(1-\beta),0}^{\mathrm{rad}}}{1-\beta}.
\]
In particular, the monotonicity in \eqref{eq:weight-monotone}
can have equality for different weights.

\section*{Statements and Declarations}

\textbf{Funding.} The authors were supported by the National
Natural Science Foundation of China (Nos.~12201089, 12371051), the
Chongqing Normal University Foundation (No.~BWQB2023006), and the
Open Research Fund of the Key Laboratory of Nonlinear Analysis and
Applications, Ministry of Education, Central China Normal University
(No.~NAA2025ORG006).

\medskip
\textbf{Competing interests.} The authors declare no competing interests.

\medskip
\textbf{Data availability.} No datasets were generated or analysed
in this theoretical study.

\medskip
\textbf{Use of AI tools.} AI tools were used for literature searches,
language editing, and assistance in formulating and checking mathematical
arguments. The authors acknowledge the use of AI tools and take full
responsibility for the correctness and integrity of all proofs.

\raggedbottom
\raggedbottom
\end{document}